\documentclass[default,12pt]{sn-jnl}
\usepackage{graphicx}%
\usepackage{multirow}%
\usepackage{amsmath,amssymb,amsfonts}%
\usepackage{amsthm}%
\usepackage{mathrsfs}%
\usepackage[title]{appendix}%
\usepackage[dvipsnames]{xcolor}%
\usepackage{textcomp}%
\usepackage{manyfoot}%
\usepackage{booktabs}%
\usepackage{algorithm}%
\usepackage{algorithmicx}%
\usepackage{algpseudocode}%
\usepackage{listings}%
\usepackage{caption}%
\usepackage{algorithm}
\usepackage{algpseudocode}
\usepackage{enumitem}
\setlist[enumerate,1]{label=(\roman*)}
\usepackage{tikz}
\usetikzlibrary{
arrows,arrows.meta,shapes,automata,backgrounds,matrix,patterns,petri,patterns.meta,positioning,calc,hobby
}
\usepackage{marvosym} 
\usepackage[alpine]{ifsym} 

\theoremstyle{thmstyleone}%
\newtheorem{theorem}{Theorem}
\newtheorem{proposition}[theorem]{Proposition}%
\newtheorem{lemma}[theorem]{Lemma}%
\newtheorem{example}[theorem]{Example}%

\theoremstyle{thmstyletwo}%

\theoremstyle{thmstylethree}%

\usepackage{pifont}

\usepackage{xifthen}

\newcommand{\PowerSet}[1]{\mathcal{P}(#1)}

\newcommand{\NodeSet}{\MacroColor{N}}
\newcommand{\EdgeSet}{\MacroColor{A}}
\newcommand{\DemandSet}{\MacroColor{Q}}
\newcommand{\NodeCutSet}{\MacroColor{C}}

\newcommand{\NodeCut}{\MacroColor{S}}

\newcommand{\DisaggNodeCutSet}[2]{%
C_{#2}^{\text{disagg}}%
}

\newcommand{\DisaggSubprobVar}[2]{%
z_{#1 #2}%
}

\newcommand{\AggNodeCutSet}[1]{%
C_{#1}^{\text{agg}}%
}

\newcommand{\FlowVolume}{\MacroColor{f}}
\newcommand{\PathSet}{\MacroColor{R}}
\newcommand{\CycleSet}{\MacroColor{R}}

\newcommand{\DemandIdx}{\MacroColor{q}}
\newcommand{\NodeIdx}{\MacroColor{j}}
\newcommand{\PathIdx}{\MacroColor{r}}
\newcommand{\CycleIdx}{\MacroColor{r}}

\newcommand{\TravelRange}{\MacroColor{d}}

\newcommand{\AggV}[1][]{%
\ifthenelse{\isempty{#1}}{v^{\text{agg}}}{v^{\text{agg}}_{#1}}%
}
\newcommand{\DisaggV}[1][]{%
\ifthenelse{\isempty{#1}}{v^{\text{disagg}}}{v^{\text{disagg}}_{#1}}%
}
\newcommand{\TightV}[1][]{%
\ifthenelse{\isempty{#1}}{v^{\text{tight}}}{v^{\text{tight}}_{#1}}%
}
\newcommand{\TighterV}[1][]{%
\ifthenelse{\isempty{#1}}{v^{\text{tighter}}}{v^{\text{tighter}}_{#1}}%
}

\newcommand{\MasterVar}{\MacroColor{x}}
\newcommand{\AggSubprobVar}{\MacroColor{y}}
\newcommand{\MasterFeasSet}{\MacroColor{X}}
\newcommand{\LprMasterFeasSet}{\MacroColor{\bar{X}}}

\newcommand{\DeviationTolerance}{\MacroColor{\alpha}}

\newcommand{\LengthLimit}{\MacroColor{\tau}}
\newcommand{\LabelCharged}{\MacroColor{\delta_{\text{charge}}}}
\newcommand{\LabelIsOutbound}{\MacroColor{\delta_{\text{dest}}}}
\newcommand{\LabelDistFromStart}{\MacroColor{l_{\text{start}}}}
\newcommand{\LabelDistFromCharger}{\MacroColor{l_{\text{charge}}}}
\newcommand{\LabelChargeAtStart}{\MacroColor{\gamma_{\text{end}}}}
\newcommand{\Nan}{\MacroColor{\infty}}
\newcommand{\ResourceExtensionFunction}{\MacroColor{f}}
\newcommand{\MyAnd}{\text{\MacroColor{and}}}
\newcommand{\MyOr}{\text{\MacroColor{or}}}
\newcommand{\EdgeLength}{\MacroColor{l}}
\newcommand{\SinkNode}{\MacroColor{s}}

\newcommand{\DistanceFunction}{\MacroColor{\text{dist}}}

\newcommand{\xmark}{\ding{55}}
\usepackage{array}
\newcolumntype{H}{>{\setbox0=\hbox\bgroup}c<{\egroup}@{}}

\newcommand{\MacroColor}[1]{#1}

\begin{document}

\title[The Strength of Flow Refueling Location Problem Formulations and an Extension to Cyclic Routing]{The Strength of Flow Refueling Location Problem Formulations and an Extension to Cyclic Routing}


\author*[1]{\fnm{Nagisa} \sur{Sugishita}}\email{nagisa.sugishita@umontreal.ca}

\author[1]{\fnm{Margarida} \sur{Carvalho}}\email{carvalho@iro.umontreal.ca}

\author[2]{\fnm{Ribal} \sur{Atallah}}\email{atallah.ribal@hydroquebec.com}

\affil*[1]{\orgdiv{CIRRELT and Departement d’informatique et de recherche operationnel}, \orgname{Université de Montréal}, \orgaddress{\street{2920 Chemin de la Tour}, \city{Montréal}, \postcode{H3T 1N8}, \state{QC}, \country{Canada}}}

\affil[2]{\orgname{Institut de Recherche d’Hydro-Quebec}, \orgaddress{\street{1800 Boulevard Lionel-Boulet}, \city{Varennes}, \postcode{J3X 1S1}, \state{QC}, \country{Canada}}}

\abstract{%
The Flow Refueling Location Problem (FRLP) is a stylized model for determining the optimal placement of refueling stations for vehicles with limited travel ranges, such as hydrogen fuel cell vehicles and electric vehicles. A notable extension, the deviation FRLP, accounts for the possibility that drivers may deviate from their preferred routes to refuel or recharge. While solution techniques based on various mathematical programming formulations have been thoroughly explored for this extension, there is a lack of theoretical insights into the relationships and strengths of these formulations. 

In this work, for the deviation extension, we study two prominent FRLP formulations from the literature and compare their strengths in terms of linear programming (LP) relaxations. We show that the LP relaxation of one formulation yields a bound that is at least as tight as that of the other, which may explain its observed superior performance. Building on these insights, we address a common modeling assumption in the FRLP that requires drivers to use the same paths for their outbound and inbound trips. Specifically, we relax this assumption and introduce the cyclic FRLP, where drivers may use different paths in each direction. We show how existing formulations can be naturally extended to accommodate this setting and describe a branch-and-cut algorithm to solve the problem. We provide numerical experiments highlighting the benefits of such asymmetric routing.
For example, in an instance based on the Californian network, the cyclic FRLP serves all demands using 30\% fewer facilities than the original FRLP.
}




\maketitle

\section{Introduction}\label{sec_introduction}

The Flow Refueling Location Problem (FRLP), introduced by \citet{KubyAndLim2005}, extends flow-capturing location models by \citet{Hodgson1990} to address the optimal placement of refueling facilities for vehicles with limited range.
In the model by \citet{KubyAndLim2005}, each demand is associated with an origin, a destination and a flow volume, and is considered served if there are enough refueling facilities to ensure that the drivers can complete their journey with limited-range vehicles. The goal of the FRLP is to maximize the demands that are served. 

On the practical level, the growing popularity of hydrogen-powered vehicles and electric vehicles (EVs), as well as governmental incentives in some countries to support their adoption, underscores the relevance of the FRLP (see \citet{ArslanAndKarasan2016} and \citet{KadriEtAl2020}). Indeed, given the current insufficiency of refueling and charging infrastructure, it is critical to deploy these facilities effectively. In what follows, we use the EVs application to unify our terminology around the FRLP, referring to charging stations for refueling facilities and EVs for limited-range vehicles.

On the methodological side, the literature on the FRLP has mostly focused on integer programming-based approaches.
\citet{KubyAndLim2005} introduced a mixed-integer programming (MIP) formulation for the FRLP.
However, its size grows rapidly with the problem scale. To address this, \citet{CaparAndKuby2012} proposed a more scalable formulation, which was further improved by \citet{CaparEtAl2013}. Building on this, \citet{ArslanAndKarasan2016} applied Benders decomposition to accelerate the solution process. \citet{KadriEtAl2020} extended the FRLP to a multi-stage stochastic setting, using a decomposition approach closely related to that of \citet{ArslanAndKarasan2016}, with an identical structure for the separation subproblem.

The previously discussed works focus on the classical FRLP, where drivers are assumed to strictly follow predetermined (e.g., shortest) paths without deviation. However, in practice, drivers may deviate slightly from their usual routes to access charging stations. To capture this behavior, \citet{KimAndKuby2012} introduced the deviation FRLP, which allows for multiple feasible paths per origin-destination (OD) pair. Their formulation requires enumerating all paths and, for each path, all combinations of station placements that ensure traversability. This leads to rapid growth in model size as the instance scale increases.
\citet{YildizEtAl2016} proposed an improved formulation of the deviation FRLP that eliminates the need to enumerate combinations of station placements.
Leveraging techniques from \citet{CaparEtAl2013}, they express path traversability using set-covering constraints.
Additionally, they introduced a path-segment-based formulation, solved via a branch-and-price algorithm. 
While this formulation, together with a branch-and-price approach, yields better performance, the set-covering formulation offers greater flexibility and modeling capabilities for incorporating additional constraints, as it directly uses binary variables indicating station presence at each node.
Independently, \citet{ArslanEtAl2019} and \citet{GopfertAndBock2019} proposed another MIP formulation, using set-covering constraints to express the condition that each demand is served, rather than expressing the condition that each path is traversable, as in \citet{YildizEtAl2016}.
This approach eliminates the need for path enumeration and, when solved with a branch-and-cut algorithm, significantly outperforms previous methods, successfully handling instances of practical size.

\begin{table}[t]
\centering
\caption{Summary of FRLP approaches}
\label{tab:frlp_models}
\begin{tabular}{lHHcc}
\toprule
 & & & Supports deviation & Allows cyclic \\
\multicolumn{1}{c}{Work} & & & from preferred route? & routing? \\
\midrule
\citet{KubyAndLim2005} & Path-based set-covering & MIP (non-compact \& heavy preprocessing) & \xmark & \xmark \\
\citet{CaparAndKuby2012}   & Range-based set-covering & MIP & \xmark & \xmark \\
\citet{CaparEtAl2013}      & Arc-based set-covering   & MIP & \xmark & \xmark \\
\citet{ArslanAndKarasan2016} & Arc-based set-covering & Benders decomposition & \xmark & \xmark \\
\citet{KimAndKuby2012}     & Path-based set-covering  & MIP (path enumeration) & \checkmark & \xmark \\
\citet{YildizEtAl2016}     & Arc-based set-covering   & MIP \& Branch-and-price & \checkmark & \xmark \\
\citet{ArslanEtAl2019} \& \citet{GopfertAndBock2019}  & Node-based covering & MIP \& Branch-and-cut & \checkmark & \xmark \\
\textbf{This paper (cyclic FRLP)}        & \textbf{Node-based covering} & \textbf{Branch-and-cut} & \checkmark & \checkmark \\
\hline
\end{tabular}
\end{table}

On the theoretical front, the comparative strengths of the formulations for the deviation FRLP remain underexplored.
We aim to address that gap.
To be more specific, we compare the MIP formulation developed independently by \citet{ArslanEtAl2019} and \citet{GopfertAndBock2019} with the formulation by \citet{YildizEtAl2016} based on the set-covering constraints.
By analyzing these formulations, we aim to understand the reasons behind the superior computational performance of the state-of-the-art formulation.
We note that although \citet{YildizEtAl2016} proposed two formulations, the set-covering formulation and path-segment formulation, we concentrate on the set-covering one, excluding the path-segment formulation due to its limited flexibility for integrating additional constraints.

In addition, we propose an extension of the deviation FRLP, which we call the cyclic FRLP, to support more flexible routing.
All prior work mentioned above assumes route symmetry, i.e., drivers must use the same path in both directions. 
While simplifying the problem, this assumption may lead to suboptimal outcomes in more realistic settings. 
Furthermore, this assumption requires the underlying transportation network to be undirected.
In the cyclic FRLP, drivers are allowed to choose different outbound and inbound routes, enabling more efficient utilization of transportation infrastructure. Notably, the formulation of \citet{ArslanEtAl2019} and \citet{GopfertAndBock2019} can be naturally extended to accommodate this generalization.
We develop a solution method for the cyclic FRLP by introducing a labeling algorithm that searches for a cycle (round-trip route) that can be repeatedly traversed by vehicles with limited range.
A summary of the literature and our contribution to the FRLP—aside from the theoretical analysis of formulation strength—is provided in Table~\ref{tab:frlp_models}.

Using the Californian network developed by \citet{arslan2014impacts}, we demonstrate the difference between the deviation and cyclic FRLP.
Furthermore, to complement our experiments and highlight the value of the cyclic FRLP, we consider a variant that minimizes the number of charging stations while covering a specified portion or all of the demand. This variant has been studied in the literature under the symmetric routing assumption (e.g., \citet{CaparEtAl2013} and \citet{KinayGzaraAndAlumur20211}) and can be readily addressed by adapting formulations such as those by \citet{ArslanEtAl2019} and \citet{GopfertAndBock2019}. Including this variant allows us to directly compare the added benefit of incorporating cyclic routes against the classical FRLP with symmetric assumptions.
For example, on one of the instances based on the Californian network by~\citet{arslan2014impacts}, the cyclic FRLP is able to serve all demands with 30\% fewer charging stations than the deviation FRLP.

The paper is organized as follows.
In Section~\ref{sec_formulations_of_frlp}, we introduce the necessary notation and present two formulations of the FRLP.
Section~\ref{sec_strength_of_formulations} compares the strengths of these formulations via their LP relaxations.
Section~\ref{sec_cyclic_frlp} introduces the cyclic FRLP, which accommodates more flexible routing.
Numerical experiments are also provided to illustrate the differences between the original and cyclic FRLP.
Finally, Section~\ref{sec_conclusions} concludes the paper.

\section{Formulations of the Deviation FRLP}\label{sec_formulations_of_frlp}

In this section, we present the arguments used to define the deviation FRLP and recap two key mathematical formulations of it from the literature.
The notation, namely the parameters and variables used in the formulations in this section, is summarized in Table~\ref{tab_symbols}.

\begin{table}[htbp]
    \centering
    \caption{Symbols and their description}
    \label{tab_symbols}
    \begin{tabular}{cl}
    \toprule
    Symbol & Description \\
    \midrule
        $\NodeSet$ & Set of nodes \\
        $\EdgeSet$ & Set of edges \\
        $G$ & Graph $(\NodeSet, \EdgeSet)$ \\
        $\DemandSet$ & Set of demands \\
        $\PathSet_{\DemandIdx}$ & Set of paths used by drivers of demand $\DemandIdx \in \DemandSet$ \\
        $\PathSet$ & Union of $\PathSet_{\DemandIdx}$ for $\DemandIdx \in \DemandSet$ \\
        $\MasterFeasSet$ & Set of feasible placement of charging stations \\
        $\LprMasterFeasSet$ & LP relaxation of $\MasterFeasSet$ \\
        $\AggNodeCutSet{\DemandIdx}$ & Family of node sets to express the condition of $\DemandIdx \in \DemandSet$ to be served \\
        $\DisaggNodeCutSet{}{\PathIdx}$ & Family of node sets to express the condition of $\PathIdx \in \PathSet$ to be traversable \\
        $\FlowVolume_{\DemandIdx}$ & Demand volume of $\DemandIdx \in \DemandSet$ \\
        $\MasterVar_{\NodeIdx}$ & Binary variable that is 1 if and only if there is a charging station at node $\NodeIdx$ \\
        $\AggSubprobVar_{\DemandIdx}$ & Binary variable that is 1 if and only if $\DemandIdx \in \DemandSet$ is served \\
        $\DisaggSubprobVar{}{\PathIdx}$ & Binary variable that is 1 if and only if $\PathIdx \in \PathSet_{\DemandIdx}$ is used \\
        \bottomrule
    \end{tabular}
\end{table}

Let $G = (\NodeSet, \EdgeSet)$ be an undirected graph, where $\NodeSet$ is the set of nodes and $\EdgeSet$ is the set of edges.
We consider the installation of charging stations in nodes of $\NodeSet$.
For each $\NodeIdx \in \NodeSet$, we use the binary variable $\MasterVar_{\NodeIdx}$ to indicate the installation of a charging station at node $\NodeIdx$.
We define the set of feasible charging station placements as
$
\MasterFeasSet = \LprMasterFeasSet \cap \{0, 1\}^{|\NodeSet|},
$
where $\LprMasterFeasSet$ is a polytope.
For example, $\MasterFeasSet$ can model charging station locations that satisfy a budget constraint.

We denote a set of demands by $\DemandSet$.
Each demand $\DemandIdx \in \DemandSet$ is associated with an origin node, a destination node, a flow volume $\FlowVolume_{\DemandIdx}$, and a set $\PathSet_{\DemandIdx}$ of paths from the origin to the destination.
The set $\PathSet_{\DemandIdx}$ is the set of paths that drivers of $\DemandIdx$ are willing to use.
Demand $\DemandIdx \in \DemandSet$ is considered served if there is at least one path in $\PathSet_\DemandIdx$ that is traversable without running out of battery.
Since the travel range of a vehicle is limited, it may require charging the EV multiple times to complete the travel.
We assume that the graph does not contain edges longer than the travel range of a vehicle and $\PathSet_\DemandIdx \not= \emptyset$ for each $\DemandIdx \in \DemandSet$.
Note that we can remove excessively long edges or demands $\DemandIdx$ with empty $\PathSet_\DemandIdx$ without modifying the set of feasible solutions, which is defined later. 

In the classical FRLP proposed by \citet{KubyAndLim2005}, each demand is associated with a single shortest path. Thus, for each $\DemandIdx \in \DemandSet$, the set $\PathSet_{\DemandIdx}$ is a singleton containing only the shortest path between the origin and destination. \citet{KimAndKuby2012} introduced the deviation FRLP, an extension of the classical FRLP in which multiple paths are considered for each origin-destination (OD) pair. 
For example, in their work, \citet{KimAndKuby2012} defined $\PathSet_{\DemandIdx}$ to include all paths whose lengths are below a specified deviation threshold, for each $\DemandIdx \in \DemandSet$.
Figure~\ref{fig_original_frlp_illustration} illustrates the problem definition.

\begin{figure}[htbp]
\centering
\begin{tikzpicture}[
xscale=1,
point/.style={circle,inner sep=0,minimum size=1mm,fill=black},
nodepadding/.style={circle,inner sep=0,minimum size=6mm},
unselected/.style={draw=black, densely dashed},
selected/.style={draw=black},
hypograph/.style={ultra thick,draw=Rhodamine},
]
\coordinate (v0) at (0.0, 0.0);
\coordinate (v1) at (5.0, 2.0);

\node[nodepadding] (n0) at (v0) {};
\node[nodepadding] (n1) at (v1) {};

\node[point,label={below:1}] at (v0) {};
\node[point,label={below right:2}] at (v1) {};

\draw[unselected] (n0) to [out=70,in=150] node [below] {} (n1);

\draw[unselected] (n0) to node [below] {} (n1);
\draw[unselected] (n0) to [bend right=45] node [below] {} (n1);

\begin{scope}[shift={(-0.16,0.36)}]
\coordinate (v0) at (0.0, 0.0);
\coordinate (v1) at (5.0, 2.0);

\node[nodepadding] (n0) at (v0) {};
\node[nodepadding] (n1) at (v1) {};

\draw[selected, -{Latex[length=2mm]}] (n0) to [out=70,in=150] (n1);
\end{scope}

\begin{scope}[shift={(-0.08,0.18)}]
\coordinate (v0) at (0.0, 0.0);
\coordinate (v1) at (5.0, 2.0);

\node[nodepadding] (n0) at (v0) {};
\node[nodepadding] (n1) at (v1) {};

\draw[selected, {Latex[length=2mm]}-] (n0) to [out=70,in=150] (n1);
\end{scope}

\end{tikzpicture}











\caption{OD pair with origin and destination at nodes $1$ and $2$, respectively; dashed lines indicate possible OD paths, and solid directed lines show a feasible route. In the classical and deviation FRLP, a demand is served if there is a path that is traversable repeatedly without a vehicle running out of battery. Drivers are required to use the same path for both outbound and inbound trips.}
\label{fig_original_frlp_illustration}
\end{figure}

As discussed by~\citet{CaparEtAl2013}, given a path and the travel range of a vehicle, one can construct a family $\NodeCutSet \subset \PowerSet{\NodeSet}$ of node sets such that the path is traversable repeatedly with the vehicle if and only if $\sum_{\NodeIdx \in \NodeCut} \MasterVar_{\NodeIdx} \ge 1$ for all $\NodeCut \in \NodeCutSet$, where $\PowerSet{\NodeSet}$ is the power set of $\NodeSet$.
In words, the path is traversable if and only if each $\NodeCut \in \NodeCutSet$ contains a node with a charging station.
Thus, for each $\DemandIdx \in \DemandSet$ and $\PathIdx \in \PathSet_{\DemandIdx}$, we can find a family $\DisaggNodeCutSet{\DemandIdx}{\PathIdx} \subset \PowerSet{\NodeSet}$ of node sets such that $\PathIdx$ is traversable if and only if $\sum_{\NodeIdx \in \NodeCut} \MasterVar_{\NodeIdx} \ge 1$ for all $\NodeCut \in \DisaggNodeCutSet{\DemandIdx}{\PathIdx}$.
Demand $\DemandIdx \in \DemandSet$ is served if and only if there exists $\PathIdx \in \PathSet_{\DemandIdx}$ such that $\sum_{\NodeIdx \in \NodeCut} \MasterVar_{\NodeIdx} \ge 1$ for all $\NodeCut \in \DisaggNodeCutSet{\DemandIdx}{\PathIdx}$.
Thus, we can obtain the following formulation of the deviation FRLP:
\begin{alignat}{2}
\max_{\MasterVar \in \MasterFeasSet, \DisaggSubprobVar{}{} \in \{0, 1\}^{|\PathSet|}} \ & \sum_{\DemandIdx \in \DemandSet} \sum_{\PathIdx \in \PathSet_{\DemandIdx}} \FlowVolume_{\DemandIdx} \DisaggSubprobVar{}{\PathIdx} \label{eq_disagg_extended_formulation} \\
\text{s.t.} \ & \sum_{\PathIdx \in \PathSet_{\DemandIdx}} \DisaggSubprobVar{}{\PathIdx} \le 1, && \qquad \forall \DemandIdx \in \DemandSet, \notag \\
& \sum_{\NodeIdx \in \NodeCut} \MasterVar_{\NodeIdx} \ge \DisaggSubprobVar{}{\PathIdx}, && \qquad \forall \DemandIdx \in \DemandSet, \PathIdx \in \PathSet_{\DemandIdx}, \NodeCut \in \DisaggNodeCutSet{\DemandIdx}{\PathIdx}, \notag
\end{alignat}
where $\PathSet = \cup_{\DemandIdx \in \DemandSet} \PathSet_{\DemandIdx}$ and, for each $\DemandIdx \in \DemandSet$, the auxiliary variable $\DisaggSubprobVar{}{\PathIdx}$ is 1 if path $\PathIdx$ is used and 0 otherwise.
Formulation~\eqref{eq_disagg_extended_formulation} was proposed by \citet{YildizEtAl2016} as formulation~(14) in their paper.
As discussed in the introduction, \citet{YildizEtAl2016} also proposed a path-segment formulation that can be solved with a branch-and-price approach.
We do not consider the path-segment formulation since it is less flexible than those considered in this section.

\begin{example}
\label{example_disagg_node_cut_set}
Consider the graph $G$ shown in Figure~\ref{fig_frlp_agg_vs_disagg}.
Suppose the travel range is $\TravelRange$, and there is a single OD pair \(\DemandIdx_{1}\), with origin and destination at nodes 1 and 4, respectively.
Assume $\PathSet_{\DemandIdx_{1}} = \{ (1, 2, 4), (1, 2, 3, 4) \}$.
Let us denote $\PathIdx_{1} = (1, 2, 4)$ and $\PathIdx_{2} = (1, 2, 3, 4)$.

\begin{figure}[htbp]
\centering
\begin{tikzpicture}[
point/.style={circle,inner sep=0,minimum size=1mm,fill=black},
hypograph/.style={ultra thick,draw=Rhodamine},
]
\coordinate (v0) at (0.0, 0 .0);
\coordinate (v1) at (3.0, 0.0);
\coordinate (v2) at (4.5, 1.5);
\coordinate (v3) at (6.0, 0.0);
\coordinate (v4) at (9.0, 0.0);

\node[point,label={below:1}] at (v0) {};
\node[point,label={below:2}] at (v1) {};
\node[point,label={above:3}] at (v2) {};
\node[point,label={below:4}] at (v3) {};
\node[point,label={below:5}] at (v4) {};

\draw 
    (v0) -- node[above]{\small $\TravelRange/2$} (v1)
    (v1) -- node[above left]{\small $\TravelRange/2$} (v2)
    (v1) -- node[above]{\small $\TravelRange/2$} (v3)
    (v2) -- node[above right]{\small $\TravelRange/2$} (v3)
    (v3) -- node[above]{\small $\TravelRange/2$} (v4)
    ;


\end{tikzpicture}
\caption{Graph $G$, where $\TravelRange$ is the travel range}
\label{fig_frlp_agg_vs_disagg}
\end{figure}

It can be shown that path \(\PathIdx_{1}\) is traversable if and only if there is at least one charging station at each of the following sets of nodes:
\begin{itemize}
\item nodes 1 or 2,
\item nodes 2 or 4,
\item nodes 4 or 5.
\end{itemize}
See \citet{CaparEtAl2013} for the derivation.
Thus, $\DisaggNodeCutSet{\DemandIdx_{1}}{\PathIdx_{1}} = \{ \{1, 2\}, \{2, 4\}, \{4, 5\} \}$.
Similarly, we have $\DisaggNodeCutSet{\DemandIdx_{1}}{\PathIdx_{2}} = \{ \{1, 2\}, \{2, 3\}, \{3, 4\}, \{4, 5\} \}$.
Therefore, formulation~\eqref{eq_disagg_extended_formulation} becomes
\begin{alignat*}{2}
f_1' = 
\max_{\MasterVar \in \MasterFeasSet, \DisaggSubprobVar{}{} \in \{0, 1\}^2} \ & 
\FlowVolume_{\DemandIdx_{1}} \DisaggSubprobVar{}{\PathIdx_{1}}
+
\FlowVolume_{\DemandIdx_{1}} \DisaggSubprobVar{}{\PathIdx_{2}}
\\
\text{s.t.} \ & 
\DisaggSubprobVar{}{\PathIdx_{1}} + \DisaggSubprobVar{}{\PathIdx_{2}} \le 1, \\
& 
\MasterVar_{1} + \MasterVar_{2} \ge \DisaggSubprobVar{}{\PathIdx_{1}},
\MasterVar_{2} + \MasterVar_{4} \ge \DisaggSubprobVar{}{\PathIdx_{1}},
\MasterVar_{4} + \MasterVar_{5} \ge \DisaggSubprobVar{}{\PathIdx_{1}},
\\
&
\MasterVar_{1} + \MasterVar_{2} \ge \DisaggSubprobVar{}{\PathIdx_{2}},
\MasterVar_{2} + \MasterVar_{3} \ge \DisaggSubprobVar{}{\PathIdx_{2}},
\MasterVar_{3} + \MasterVar_{4} \ge \DisaggSubprobVar{}{\PathIdx_{2}},
\MasterVar_{4} + \MasterVar_{5} \ge \DisaggSubprobVar{}{\PathIdx_{2}}.
\end{alignat*}
\end{example}

Next, we derive an alternative formulation.
While this formulation has been previously introduced, we provide a new and straightforward derivation here to keep the exposition self-contained.
To this end, let us look at another way to formulate the condition of demand $\DemandIdx \in \DemandSet$ being served.
For $\DemandIdx \in \DemandSet$, let $\AggNodeCutSet{\DemandIdx} = \{ \cup_{\PathIdx \in \PathSet_{\DemandIdx}} \NodeCut_{\PathIdx} : \NodeCut_{\PathIdx} \in \DisaggNodeCutSet{\DemandIdx}{\PathIdx}, \PathIdx \in \PathSet_{\DemandIdx} \}$, i.e., an element of $\AggNodeCutSet{\DemandIdx}$ is the union of 
one $S_r$ for each $\PathIdx \in \PathSet$. 
The set $\AggNodeCutSet{\DemandIdx}$ has the following property.

\begin{lemma}
\label{lemma_correctness_of_aggregation}
Demand $\DemandIdx \in \DemandSet$ is served if and only if $\sum_{\NodeIdx \in \NodeCut} \MasterVar_{\NodeIdx} \ge 1$ for all $\NodeCut \in \AggNodeCutSet{\DemandIdx}$.
\end{lemma}

\begin{proof}
Suppose that $\DemandIdx \in \DemandSet$ is served.
Then, there exists $\PathIdx' \in \PathSet_{\DemandIdx}$ such that
\begin{equation*}
\sum_{\NodeIdx \in \NodeCut'} \MasterVar_{\NodeIdx} \ge 1, \qquad
\forall \NodeCut' \in \DisaggNodeCutSet{\DemandIdx}{\PathIdx'}.
\end{equation*}
Let $\NodeCut \in \AggNodeCutSet{\DemandIdx}$.
By definition, we can write $\NodeCut$ as
\begin{equation*}
\NodeCut = \bigcup_{\PathIdx \in \PathSet_{\DemandIdx}} \NodeCut_{\PathIdx},
\end{equation*}
where $\NodeCut_{\PathIdx} \in \DisaggNodeCutSet{\DemandIdx}{\PathIdx}$ for each $\PathIdx \in \PathSet_{\DemandIdx}$.
It follows that
\begin{equation*}
\sum_{\NodeIdx \in \NodeCut} \MasterVar_{\NodeIdx}
\ge
\sum_{\NodeIdx \in \NodeCut_{\PathIdx'}} \MasterVar_{\NodeIdx}
\ge 1.
\end{equation*}

Next, suppose that $\DemandIdx \in \DemandSet$ is not served.
For any $\PathIdx \in \PathSet_{\DemandIdx}$, there exists $\NodeCut_{\PathIdx}' \in \DisaggNodeCutSet{\DemandIdx}{\DemandIdx}$ such that
\begin{equation*}
\sum_{\NodeIdx \in \NodeCut_{\PathIdx}'} \MasterVar_{\NodeIdx} = 0.
\end{equation*}
Let 
\begin{equation*}
\NodeCut' = \bigcup_{\PathIdx \in \PathSet_{\DemandIdx}} \NodeCut_{\PathIdx}'.
\end{equation*}
Then, $\NodeCut' \in \AggNodeCutSet{\DemandIdx}$ and
\begin{equation*}
\sum_{\NodeIdx \in \NodeCut'} \MasterVar_{\NodeIdx}
\le
\sum_{\PathIdx \in \PathSet_{\DemandIdx}} \sum_{\NodeIdx \in \NodeCut_{\PathIdx}} \MasterVar_{\NodeIdx}
= 0.
\end{equation*}
\end{proof}

Therefore, we obtain the second formulation of the deviation FRLP for maximizing the amount of served demands as
\begin{subequations}
\label{eq_agg_extended_formulation}
\begin{alignat}{2}
\max_{\MasterVar \in \MasterFeasSet, \AggSubprobVar \in \{0, 1\}^{|\DemandSet|}} \ & \sum_{\DemandIdx \in \DemandSet} \FlowVolume_{\DemandIdx} \AggSubprobVar_{\DemandIdx} \\
\text{s.t.} \ & \sum_{\NodeIdx \in \NodeCut} \MasterVar_{\NodeIdx} \ge \AggSubprobVar_{\DemandIdx}, && \qquad \forall \DemandIdx \in \DemandSet, \NodeCut \in \AggNodeCutSet{\DemandIdx}, 
\label{eq_agg_extended_formulation_constraint}
\end{alignat}
\end{subequations}
where $\AggSubprobVar_{\DemandIdx}$ is $1$ if demand $\DemandIdx \in \DemandSet$ is served, and 0 otherwise.
This is the formulation used by \citet{ArslanEtAl2019} and \citet{GopfertAndBock2019}.
Note that when $\PathSet_{\DemandIdx}$ is a singleton for each $\DemandIdx \in \DemandSet$, formulations~\eqref{eq_disagg_extended_formulation} and~\eqref{eq_agg_extended_formulation} coincide and are equivalent to the one used by \citet{CaparEtAl2013}.

\begin{example}
\label{example_agg_node_cut_set}
Let us consider the graph $G$ from Example~\ref{example_disagg_node_cut_set} again.  
We have
\begin{align}
\AggNodeCutSet{\DemandIdx_{1}} = \Big\{ &
\{1, 2\}, \{1, 2, 3\}, \{1, 2, 3, 4\}, \{1, 2, 4, 5\}, \{1, 2, 4\}, \{2, 3, 4\}, 
\notag
\\
& \quad 
\{2, 4, 5\}, 
\{2, 3, 4, 5\}, \{3, 4, 5\}, \{4, 5\} \Big\}.
\label{eq_example_agg_node_cut_set}
\end{align}
Therefore, formulation~\eqref{eq_agg_extended_formulation} is
\begin{alignat*}{2}
f_2' = 
\max_{\MasterVar \in \MasterFeasSet, \AggSubprobVar \in \{0, 1\}} \ & 
\FlowVolume_{\DemandIdx_{1}} \AggSubprobVar_{\DemandIdx_{1}} 
\\
\text{s.t.} \ 
& \MasterVar_{1} + \MasterVar_{2} \ge \AggSubprobVar_{\DemandIdx_{1}}, 
\MasterVar_{1} + \MasterVar_{2} + \MasterVar_{3} \ge \AggSubprobVar_{\DemandIdx_{1}},
\MasterVar_{1} + \MasterVar_{2} + \MasterVar_{3} + \MasterVar_{4} \ge \AggSubprobVar_{\DemandIdx_{1}}, \\
& \MasterVar_{1} + \MasterVar_{2} + \MasterVar_{4} + \MasterVar_{5} \ge \AggSubprobVar_{\DemandIdx_{1}}, 
\MasterVar_{1} + \MasterVar_{2} + \MasterVar_{4} \ge \AggSubprobVar_{\DemandIdx_{1}},
\MasterVar_{2} + \MasterVar_{3} + \MasterVar_{4} \ge \AggSubprobVar_{\DemandIdx_{1}}, \\
& \MasterVar_{2} + \MasterVar_{4} + \MasterVar_{5} \ge \AggSubprobVar_{\DemandIdx_{1}}, 
\MasterVar_{2} + \MasterVar_{3} + \MasterVar_{4} + \MasterVar_{5} \ge \AggSubprobVar_{\DemandIdx_{1}},
\MasterVar_{3} + \MasterVar_{4} + \MasterVar_{5} \ge \AggSubprobVar_{\DemandIdx_{1}}, \\
& \MasterVar_{4} + \MasterVar_{5} \ge \AggSubprobVar_{\DemandIdx_{1}}.
\end{alignat*}
\end{example}

Formulation~\eqref{eq_agg_extended_formulation} has fewer variables than formulation~\eqref{eq_disagg_extended_formulation}, but many constraints since
\begin{equation*}
|\AggNodeCutSet{\DemandIdx}| = \prod_{\PathIdx \in \PathSet_{\DemandIdx}} |\DisaggNodeCutSet{\DemandIdx}{\PathIdx}|,
\end{equation*}
which grows rapidly as the number of paths in $\PathSet_{\DemandIdx}$ increases.
However, some of the constraints are redundant.
For instance, in Example~\ref{example_agg_node_cut_set}, the covering constraints for the subset 
$
\{\{1, 2\}, \{2, 3, 4\}, \{4, 5\}\}
$
of $\AggNodeCutSet{\DemandIdx_{1}}$ imply those of $\AggNodeCutSet{\DemandIdx_{1}}$. 
Note that this subset contains minimal node sets, i.e., those that are not strict supersets of any other node sets in  $\AggNodeCutSet{\DemandIdx_{1}}$.
We can therefore formalize this reasoning.
To that end, let us define
\begin{align*}
\bar{P}(C) = \left\{ (x, y) \in [0, 1]^{|\NodeSet| + 1} : \sum_{\NodeIdx \in \NodeCut} \MasterVar_{\NodeIdx} \ge y, \ \ \forall \NodeCut \in C \right\}
\end{align*}
where $C \subset \PowerSet{\NodeSet}$ is a family of node sets.

\begin{proposition}
\label{prop_minimal_node_set}
Let $C \subset \PowerSet{\NodeSet}$ be a family of node sets and $C' \subset C$ denote the subfamily of minimal ones:
$
C' = \{ \NodeCut' \in C : \NodeCut = \NodeCut' \ \MyOr \ \NodeCut \not\subset \NodeCut', \ \forall \NodeCut \in C \}.
$
\begin{enumerate}
\item 
It holds $\bar{P}(C) = \bar{P}(C')$.
\item 
Let $\NodeCut \in C'$, then 
there exists $(x, y) \in \{0, 1\}^{|\NodeSet| + 1}$ such that $(x, y) \in \bar{P}(C \backslash \{ \NodeCut \}) \backslash \bar{P}(C)$.
\item 
Let $\NodeCut \in C'$, then
$\sum_{\NodeIdx \in \NodeCut} \MasterVar_{\NodeIdx} \ge y_{\DemandIdx}$ is facet-defining for $\mathrm{conv}(\bar{P}(C) \cap \{0, 1\}^{|\NodeSet| + 1})$.
\end{enumerate}
\end{proposition}

\begin{proof}
\begin{enumerate}
\item
It is clear that $\bar{P}(C) \subset \bar{P}(C')$ since $C' \subset C$, so we will show $\bar{P}(C) \supset \bar{P}(C')$.
To derive a contradiction, suppose that $\bar{P}(C) \not\supset \bar{P}(C')$ and let $(\MasterVar, \AggSubprobVar) \in \bar{P}(C') \backslash \bar{P}(C)$.
Since $(\MasterVar, \AggSubprobVar) \in \bar{P}(C')$, $0 \le x \le 1$ and $0 \le y \le 1$.
Furthermore, since $(\MasterVar, \AggSubprobVar) \not\in \bar{P}(C)$, there exists $\NodeCut \in C$ such that $\sum_{\NodeIdx \in \NodeCut} \MasterVar_{\NodeIdx} < \AggSubprobVar$.
However, we can find $\NodeCut' \in C'$ such that $\NodeCut' \subset \NodeCut$, for which we have $\sum_{\NodeIdx \in \NodeCut'} \MasterVar_{\NodeIdx} \le \sum_{\NodeIdx \in \NodeCut} \MasterVar_{\NodeIdx} < \AggSubprobVar$.
This contradicts the assumption that $(\MasterVar, \AggSubprobVar) \in \bar{P}(C')$.
Thus, it follows that $\bar{P}(C) \supset \bar{P}(C')$.
\item 
Let $S$ be any node set in $C'$.
Define
\begin{equation*}
\MasterVar_{\NodeIdx} = \begin{cases}
0 & \text{if $\NodeIdx \in S$,} \\
1 & \text{otherwise.} \\
\end{cases}
\end{equation*}
Then, for any $S' \in C \backslash \{S\}$, we have $S' \backslash S \not= \emptyset$ (otherwise $S$ is not minimal and $S \not\in C'$) and thus
\begin{equation*}
\sum_{\NodeIdx \in S'} \MasterVar_{\NodeIdx} \ge 1.
\end{equation*}
Therefore, $(\MasterVar, 1) \in \bar{P}(C \backslash \{S\})$.
However, since
\begin{equation*}
\sum_{\NodeIdx \in S} \MasterVar_{\NodeIdx} = 0,
\end{equation*}
we have $(\MasterVar, 1) \not\in \bar{P}(C)$.
\item 
Observe that $\mathrm{conv}(\bar{P}(C) \cap \{0, 1\}^{|\NodeSet| + 1})$ is full dimensional and that the inequality is valid for $\mathrm{conv}(\bar{P}(C) \cap \{0, 1\}^{|\NodeSet|})$.
The following set contains $|\NodeSet| + 1$ affinely independent points that satisfy the inequality with equality:
\begin{equation*}
\Bigg\{ (0, 0) \Bigg\} \cup
\Bigg\{ (e_{\NodeIdx}, 0) : \NodeIdx \in \NodeSet \backslash S \Bigg\} \cup
\Bigg\{ \Bigg(e_{\NodeIdx} + \sum_{\NodeIdx' \in \NodeSet \backslash S} e_{\NodeIdx'}, 1 \Bigg) : \NodeIdx \in S \Bigg\},
\end{equation*}
where $e_{\NodeIdx} \in \mathbb{R}^{|\NodeSet|}$ is the $\NodeIdx$-th unit vector.
Therefore, the given inequality is facet-defining.
\end{enumerate}
\end{proof}

In particular, Proposition~\ref{prop_minimal_node_set} implies that, for $\DemandIdx \in \DemandSet$, we can drop the node sets from $\AggNodeCutSet{\DemandIdx}$ that are not minimal, reducing the number of constraints~\eqref{eq_agg_extended_formulation_constraint}.
However, dropping a minimal node set from $\AggNodeCutSet{\DemandIdx}$ may introduce extra integer points.
We also note that Assertion~(iii) in Proposition~\ref{prop_minimal_node_set} concerns the polytope $\bar{P}(\AggNodeCutSet{\DemandIdx})$, which is associated with a single OD pair $\DemandIdx \in \DemandSet$. 
Proposition~6 by \citet{ArslanEtAl2019} gives the condition under which constraint~\eqref{eq_agg_extended_formulation_constraint} is facet-defining for the convex hull of the feasible region of formulation~\eqref{eq_agg_extended_formulation}.

For $\DemandIdx \in \DemandSet$, when $\PathSet_{\DemandIdx}$ consists of a single, simple path, the polytope $\bar{P}(\AggNodeCutSet{\DemandIdx})$ has the following property.

\begin{proposition}
\label{prop_integrality_of_set_covering_polytope}
Let $\DemandIdx \in \DemandSet$.
Suppose $\PathSet_{\DemandIdx}$ consists of a single, simple path.
Then, $\bar{P}(\AggNodeCutSet{\DemandIdx})$ is integral (i.e.\ all the vertexes of $\bar{P}(\AggNodeCutSet{\DemandIdx})$ are integral).
\end{proposition}

To show Proposition~\ref{prop_integrality_of_set_covering_polytope}, we first present two lemmas.

\begin{lemma}
\label{lemma_PC_vertexes_has_0_1_y}
Let $C \subset \PowerSet{\NodeSet}$ be a family of node sets and $(\MasterVar^*, \AggSubprobVar^*)$ be a vertex of $\bar{P}(C)$.
Then, we have $\AggSubprobVar^* \in \{0, 1\}$.
\end{lemma}

\begin{proof}
Let us rewrite the definition of polyhedron $\bar{P}(C)$ as
\begin{equation*}
\bar{P}(C) = \left\{
(\MasterVar, \AggSubprobVar) \in \mathbb{R}^{|\NodeSet| + 1} : 
\begin{array}{lll}
\sum_{\NodeIdx \in \NodeCut} \MasterVar_{\NodeIdx} \ge \AggSubprobVar, & \NodeCut \in C, \\
\MasterVar_{\NodeIdx} \ge 0, & \NodeIdx \in \NodeSet, \\
\MasterVar_{\NodeIdx} \le 1, & \NodeIdx \in \NodeSet, \\
a_{\NodeIdx} \AggSubprobVar \ge b_{\NodeIdx}, & \NodeIdx \in \{1, 2\}
\end{array}
\right\},
\end{equation*}
where $a = (1, -1)$ and $b = (0, -1)$.
There are $|\NodeSet| + 1$ inequalities that are active at $(\MasterVar^*, \AggSubprobVar^*)$ and whose coefficients are linearly independent.
That is, we can find $I_{1} \subset C$, $I_{2}$, $I_{3} \subset \NodeSet$ and $I_{4} \subset \{1, 2\}$ such that $|I_{1}| + |I_{2}| + |I_{3}| + |I_{4}| = |\NodeSet| + 1$ and $(\MasterVar^*, \AggSubprobVar^*)$ is a unique solution to
\begin{equation}
\left\{
\begin{array}{ll}
\sum_{\NodeIdx \in \NodeCut} \MasterVar_{\NodeIdx} = \AggSubprobVar, & \NodeCut \in I_{1}, \\
\MasterVar_{\NodeIdx} = 0, & \NodeIdx \in I_{2}, \\
\MasterVar_{\NodeIdx} = 1, & \NodeIdx \in I_{3}, \\
a_{\NodeIdx} \AggSubprobVar = b_{\NodeIdx}, & \NodeIdx \in I_{4}.
\end{array}
\right.
\label{lemma_PC_vertexes_has_0_1_y_vertex_defining_system}
\end{equation}
If $1 \in I_{4}$, $y \ge 0$ is active at $(x^*, y^*)$ and $y^* = 0$.
Similarly, if $2 \in I_{4}$, $y \le 1$ is active  at $(x^*, y^*)$ and $y^* = 1$.
In the following, suppose that $I_{4} = \emptyset$.

If $I_{3} \cap S = \emptyset$ for any $S \in I_{1}$, then
\begin{equation*}
\MasterVar_{\NodeIdx} = \begin{cases}
1 & \text{ if } i \in I_{3}, \\
0 & \text{ otherwise, }
\end{cases}
\qquad
\AggSubprobVar = 0
\end{equation*}
is a solution to~\eqref{lemma_PC_vertexes_has_0_1_y_vertex_defining_system}.
Since $(\MasterVar^*, \AggSubprobVar^*)$ is the unique solution to~\eqref{lemma_PC_vertexes_has_0_1_y_vertex_defining_system}, we have $\AggSubprobVar^* = 0$.

If there exists $S \in I_{1}$ such that $I_{3} \cap S \not= \emptyset$ , then
\begin{equation*}
\AggSubprobVar^* = \sum_{\NodeIdx \in S} \MasterVar_{\NodeIdx}^* \ge \sum_{\NodeIdx \in S \cap I_{3}} \MasterVar_{\NodeIdx}^* = |S \cap I_{3}| \ge 1,
\end{equation*}
where we used the assumption that $\sum_{\NodeIdx \in S} \MasterVar_{\NodeIdx} \ge \AggSubprobVar$ and $\MasterVar_{\NodeIdx} \le 1$ for $\NodeIdx \in I_{3}$ are active at $(\MasterVar^*, \AggSubprobVar^*)$.
Since $(\MasterVar^*, \AggSubprobVar^*) \in \bar{P}(C)$, we have $\AggSubprobVar^* \le 1$.
Therefore, $\AggSubprobVar^* = 1$.
\end{proof}

\begin{lemma}
\label{lemma_total_unimodularity_of_set_covering_polytope}
Let $\DemandIdx \in \DemandSet$ and define
\begin{equation*}
U(\AggNodeCutSet{\DemandIdx}) = \left\{ x \in [0, 1]^{|N|} : \sum_{j \in S} x_j \ge 1, \ \ \forall S \in \AggNodeCutSet{\DemandIdx} \right\}.
\end{equation*}
If $\PathSet_{\DemandIdx}$ contains a single, simple path, $U(\AggNodeCutSet{\DemandIdx})$ is integral.
\end{lemma}

\begin{proof}
By relabeling the nodes $\NodeSet$ if necessary, we can assume that $\NodeSet = \{1, 2, \ldots, |\NodeSet|\}$ and the path in $\PathSet_{\DemandIdx}$ is $(1, 2, \ldots, t)$.
The argument of~\citet{CaparEtAl2013} ensures that any $\NodeCut \in \AggNodeCutSet{\DemandIdx}$ is a set of consecutive integers.
That is, if we write
\begin{equation*}
U(\AggNodeCutSet{\DemandIdx}) = \left\{ x \in [0, 1]^{|N|} : B x \ge 1 \right\},
\end{equation*}
where $B \in \{0, 1\}^{|\AggNodeCutSet{\DemandIdx}| \times |\NodeSet|}$, $1$s in each row of $B$ occur
in consecutive positions.
It is known that $B$ is totally unimodular (see \citet{FulkersonAndGross1965}).
It is straightforward to see that all vertexes of $U(\AggNodeCutSet{\DemandIdx})$ are integral.
\end{proof}

\begin{proof}[Proof of Proposition~\ref{prop_integrality_of_set_covering_polytope}]
Let us denote $\{0, 1\}^{|\NodeSet|} = \{\beta_i : i = 1, 2, \ldots, 2^{|\NodeSet|} \}$ and the set of vertexes of a poyhedron $P$ by $\mathrm{V}(P)$.
Then, Lemma~\ref{lemma_PC_vertexes_has_0_1_y} implies that
\begin{equation*}
\mathrm{V}(\bar{P}(\AggNodeCutSet{\DemandIdx})) = \{ (\beta_i, 0) : i = 1, 2, \ldots, 2^{|N|} \} \cup \{ (x, 1) : x \in \mathrm{V}(U(\AggNodeCutSet{\DemandIdx})) \}.
\end{equation*}
Now invoke Lemma~\ref{lemma_total_unimodularity_of_set_covering_polytope} to show that $\mathrm{V}(U(\AggNodeCutSet{\DemandIdx}))$ only contains integral points.
\end{proof}

\section{Strength of Formulations of the Deviation FRLP}\label{sec_strength_of_formulations}

In Section~\ref{sec_formulations_of_frlp}, two formulations of the deviation FRLP were presented.
In this section, we compare their strengths in terms of the tightness of the bounds given by their LP relaxations.

To facilitate the discussion, we rewrite formulations~\eqref{eq_disagg_extended_formulation} and~\eqref{eq_agg_extended_formulation}.
Following the arguments of \citet{KubyAndLim2005}, it can be shown that the optimal objective value of~\eqref{eq_disagg_extended_formulation} remains unchanged if the integrality constraints on $\DisaggSubprobVar{}{\PathIdx}$ are relaxed.
Thus, we can rewrite formulation~\eqref{eq_disagg_extended_formulation} as
\begin{equation}
\label{eq_disagg_value_function_formulation}
\max_{\MasterVar \in \MasterFeasSet} \sum_{\DemandIdx \in \DemandSet} \DisaggV[\DemandIdx](\MasterVar),
\end{equation}
where for each $\DemandIdx \in \DemandSet$ and $\MasterVar \in [0, 1]^{|\NodeSet|}$, we have
\footnote{The vector $z_{\DemandIdx} \in [0, 1]^{|R_q|}$ in $\DisaggV[\DemandIdx]$ is not the same as $z \in [0, 1]^{|R|}$ in~\eqref{eq_disagg_extended_formulation}, but both reflect the traversability of the corresponding paths.}
\begin{equation}
\DisaggV[\DemandIdx](x) = \max_{\DisaggSubprobVar{\DemandIdx}{} \in [0, 1]^{|\PathSet_{\DemandIdx}|}} \left\{ \sum_{\PathIdx \in \PathSet_{\DemandIdx}} \FlowVolume_{\DemandIdx} \DisaggSubprobVar{\DemandIdx}{\PathIdx} : \sum_{\PathIdx \in \PathSet_{\DemandIdx}} \DisaggSubprobVar{\DemandIdx}{\PathIdx} \le 1, \sum_{\NodeIdx \in \NodeCut} \MasterVar_{\NodeIdx} \ge \DisaggSubprobVar{\DemandIdx}{\PathIdx}, \ \ \forall \PathIdx \in \PathSet_{\DemandIdx},  \NodeCut \in \DisaggNodeCutSet{\DemandIdx}{\PathIdx} \right\}.
\label{eq_disagg_formulation}
\end{equation}
We have $\MasterVar \in \MasterFeasSet$ is optimal for~\eqref{eq_disagg_extended_formulation} if and only if it is optimal for~\eqref{eq_disagg_value_function_formulation}.
Note that being the value function of an LP, $\DisaggV[\DemandIdx]$ is concave for each $\DemandIdx \in \DemandSet$.

Similarly, it can be shown that the optimal objective value does not change if we relax the integrality constraint of $\AggSubprobVar_{\DemandIdx}$ for $\DemandIdx \in \DemandSet$ in formulation~\eqref{eq_agg_extended_formulation}.
Therefore, we can rewrite formulation~\eqref{eq_agg_extended_formulation} as
\begin{equation}
\label{eq_agg_value_function_formulation}
\max_{\MasterVar \in \MasterFeasSet} \sum_{\DemandIdx \in \DemandSet} \AggV[\DemandIdx](\MasterVar),
\end{equation}
where, for each $\DemandIdx \in \DemandSet$ and $\MasterVar \in [0, 1]^{|\NodeSet|}$, we have
\begin{equation}
\AggV[\DemandIdx](\MasterVar) = \max_{\AggSubprobVar_{\DemandIdx} \in [0, 1]} \left\{ \FlowVolume_{\DemandIdx} \AggSubprobVar_{\DemandIdx} : \sum_{\NodeIdx \in \NodeCut} \MasterVar_{\NodeIdx} \ge \AggSubprobVar_{\DemandIdx}, \ \ \forall \NodeCut \in \AggNodeCutSet{\DemandIdx} \right\}.
\label{eq_agg_formulation}
\end{equation}
Again, $\MasterVar \in \MasterFeasSet$ is optimal for~\eqref{eq_agg_extended_formulation} if and only if it is optimal for~\eqref{eq_agg_value_function_formulation}.
Note that $\AggV[\DemandIdx]$ is a concave function for each $\DemandIdx \in \DemandSet$.

For any $\MasterVar \in \{0, 1\}^{|\NodeSet|}$ and $\DemandIdx \in \DemandSet$, we have
\begin{equation*}
\DisaggV[\DemandIdx](\MasterVar) = \AggV[\DemandIdx](\MasterVar) = \begin{cases}
\FlowVolume_{\DemandIdx} & \text{if demand $\DemandIdx$ is served}, \\
0 & \text{otherwise},
\end{cases}
\end{equation*}
and thus
\begin{equation*}
\max_{\MasterVar \in \MasterFeasSet} \sum_{\DemandIdx \in \DemandSet} \DisaggV[\DemandIdx](x)
=
\max_{\MasterVar \in \MasterFeasSet} \sum_{\DemandIdx \in \DemandSet} \AggV[\DemandIdx](x).
\end{equation*}
However, functions $\DisaggV[\DemandIdx]$ and $\AggV[\DemandIdx]$ may not agree on $\MasterVar \in \LprMasterFeasSet$ (i.e.\ if $\MasterVar$ is fractional), and in particular, the LP relaxations of~\eqref{eq_disagg_extended_formulation} and~\eqref{eq_agg_extended_formulation} may give different upper bounds.

In general, for each $\DemandIdx \in \DemandSet$, if we have a concave function $v_{\DemandIdx}$ on $[0, 1]^{|\NodeSet|}$ that agrees with $\AggV[\DemandIdx]$ on $\{0, 1\}^{|\NodeSet|}$, we can replace $\AggV[\DemandIdx]$ with $v_{\DemandIdx}$ without affecting the optimal objective value nor the structure of the problem:
\begin{equation}
\max_{\MasterVar \in \MasterFeasSet} \sum_{\DemandIdx \in \DemandSet} \AggV[\DemandIdx](x)
=
\max_{\MasterVar \in \MasterFeasSet} \sum_{\DemandIdx \in \DemandSet} v_{\DemandIdx}(x).
\label{eq_invariance_of_solution_under_K}
\end{equation}
Let us define
\begin{equation}
K_{\DemandIdx} = \left\{ v : \text{$v$ is a concave function on $[0, 1]^{|\NodeSet|}$} \ \MyAnd \ \AggV[\DemandIdx](\MasterVar') = v(\MasterVar'), \ \  \forall \MasterVar' \in \{0, 1\}^{|\NodeSet|} \right\}.
\label{eq_def_K_q}
\end{equation}
We have $K_q \not= \emptyset$ since $\AggV[\DemandIdx] \in K_q$.
Now, let
\begin{equation*}
\TightV[\DemandIdx](\MasterVar) = \inf_{v \in K_q} \{ v(\MasterVar) \}.
\end{equation*}
Being the point-wise infimum of concave functions, $\TightV[\DemandIdx]$ is concave and thus $\TightV[\DemandIdx] \in K_q$.
The following lemma gives an alternative definition of $\TightV[\DemandIdx]$.
\begin{lemma}
\label{lemma_tightv_formula}
Let us denote $\{0, 1\}^{|\NodeSet|} = \{\beta_i : i = 1, 2, \ldots, 2^{|\NodeSet|} \}$, that is, $\beta_i \in \{0, 1\}^{|\NodeSet|}$ for each $i$ and $\beta_i \not= \beta_j$ for any $i \not= j$.
Then, we have
\begin{equation*}
\TightV[\DemandIdx](\MasterVar) = \max_{\alpha} \left\{ \sum_{i = 1}^{2^{|\NodeSet|}} \alpha_i \AggV_{\DemandIdx}(\beta_i) 
:
\sum_{i = 1}^{2^{|\NodeSet|}} \alpha_i \beta_i = \MasterVar,
\sum_{i = 1}^{2^{|\NodeSet|}} \alpha_i = 1, \alpha \ge 0
\right\}.
\end{equation*}
\end{lemma}

\begin{proof}
Let us write
\begin{equation*}
v_{\DemandIdx}(\MasterVar) = \max_{\alpha} \left\{ \sum_{i = 1}^{2^{|\NodeSet|}} \alpha_i \AggV_{\DemandIdx}(\beta_i) 
:
\sum_{i = 1}^{2^{|\NodeSet|}} \alpha_i \beta_i = \MasterVar,
\sum_{i = 1}^{2^{|\NodeSet|}} \alpha_i = 1, \alpha \ge 0
\right\},
\end{equation*}
where $v_{\DemandIdx}$ is polyhedral and concave since it is the value function of an LP.
Furthermore, $v_{\DemandIdx}(\MasterVar) = \AggV[\DemandIdx](\MasterVar)$ for any $\MasterVar \in \{0, 1\}^{|\NodeSet|}$.
We show that for any $v_{\DemandIdx}' \in K_{\DemandIdx}$, we have $v_{\DemandIdx} \le v_{\DemandIdx}'$, which implies that $\TightV[\DemandIdx] = v_{\DemandIdx}$.
Fix any $v_{\DemandIdx}' \in K_{\DemandIdx}$.
To derive a contradiction, suppose there exists $\MasterVar' \in [0, 1]^{|\NodeSet|}$ such that $v_{\DemandIdx}(\MasterVar') > v_{\DemandIdx}'(\MasterVar')$.
Let $\alpha'$ be the solution of the LP in the definition in $v_{\DemandIdx}$ given $\MasterVar = \MasterVar'$.
In particular, it follows that
\begin{equation*}
\sum_{i = 1}^{2^{|\NodeSet|}} \alpha_i' \beta_i = \MasterVar'.
\end{equation*}
We have
\begin{align*}
v_{\DemandIdx}'(\MasterVar')
< v_{\DemandIdx}(\MasterVar')
= \sum_{i = 1}^{2^{|\NodeSet|}} \alpha_i' \AggV[\DemandIdx](\beta_i)
= \sum_{i = 1}^{2^{|\NodeSet|}} \alpha_i' v_{\DemandIdx}'(\beta_i),
\end{align*}
contradicting the assumption that $v_{\DemandIdx}'$ is concave.
Therefore, we have $v_{\DemandIdx} \le v_{\DemandIdx}'$.
\end{proof}

In particular, Lemma~\ref{lemma_tightv_formula} implies that $\TightV[\DemandIdx]$ is polyhedral for any $\DemandIdx \in \DemandSet$.

The set $K_{\DemandIdx}$ consists of concave functions that model the amount of served demand corresponding to $\DemandIdx \in \DemandSet$.
The tightest function in $K_q$ is $\TightV[\DemandIdx]$ in the sense that it gives the best continuous (LP) relaxations among $K_{\DemandIdx}$:
\begin{equation*}
\max_{\MasterVar \in \LprMasterFeasSet} \sum_{\DemandIdx \in \DemandSet} \TightV[\DemandIdx](x)
\le
\max_{\MasterVar \in \LprMasterFeasSet} \sum_{\DemandIdx \in \DemandSet} v_{\DemandIdx}(x), \ \  \forall v_{\DemandIdx} \in K_{\DemandIdx}, \DemandIdx \in \DemandSet.
\end{equation*}

Now we are ready to state our first comparison result.

\begin{proposition}
\label{prop_comparison_of_v_on_general_point}
Let $\DemandIdx \in \DemandSet$ and $\MasterVar \in [0, 1]^{|\NodeSet|}$.
\begin{enumerate}
\item 
It holds
\begin{equation*}
\TightV[\DemandIdx](\MasterVar) \le \AggV[\DemandIdx](\MasterVar) \le \DisaggV[\DemandIdx](\MasterVar).
\end{equation*}
\item 
If $\PathSet_{\DemandIdx}$ consists of a single path, we have
\begin{equation*}
\AggV[\DemandIdx](\MasterVar) = \DisaggV[\DemandIdx](\MasterVar).
\end{equation*}
\item 
If $\PathSet_{\DemandIdx}$ consists of a single, simple path, we have
\begin{equation}
\label{eq_tightness_of_two_formulations_under_simple_paths_assumption}
\TightV[\DemandIdx](\MasterVar) = \AggV[\DemandIdx](\MasterVar) = \DisaggV[\DemandIdx](\MasterVar).
\end{equation}
\end{enumerate}
\end{proposition}

\begin{proof}
\begin{enumerate}
\item 
The inequality $\TightV[\DemandIdx](\MasterVar) \le \AggV[\DemandIdx](\MasterVar)$ follows from the definition, thus we will show $\AggV[\DemandIdx](\MasterVar) \le \DisaggV[\DemandIdx](\MasterVar)$.
If $\DisaggV[\DemandIdx](\MasterVar) = \FlowVolume_{\DemandIdx}$, the inequality follows from $\AggV[\DemandIdx](\MasterVar) \le \FlowVolume_{\DemandIdx}$.
Thus, assume $\DisaggV[\DemandIdx](\MasterVar) < \FlowVolume_{\DemandIdx}$.
Let $\DisaggSubprobVar{}{} \in [0, 1]^{|\PathSet|}$ and $\AggSubprobVar_{\DemandIdx} \in [0, 1]$ be the solution of the LPs in~\eqref{eq_disagg_formulation} and~\eqref{eq_agg_formulation}.
Since $\DisaggV[\DemandIdx](\MasterVar) < \FlowVolume_{\DemandIdx}$, we have
\begin{equation*}
\sum_{\PathIdx \in \PathSet_{\DemandIdx}} \DisaggSubprobVar{\DemandIdx}{\PathIdx} < 1.
\end{equation*}
For each $\PathIdx \in \PathSet_{\DemandIdx}$, there exists $\NodeCut_{\DemandIdx \PathIdx} \in \DisaggNodeCutSet_{\DemandIdx}$ such that
\begin{equation*}
\sum_{\NodeIdx \in \NodeCut_{\DemandIdx \PathIdx}} \MasterVar_{\NodeIdx} = \DisaggSubprobVar{\DemandIdx}{\PathIdx},
\end{equation*}
since otherwise we can increase $\DisaggSubprobVar{\DemandIdx}{\PathIdx}$ and improve the objective value.
Now, let
\begin{equation*}
\NodeCut_{\DemandIdx} = \bigcup_{\PathIdx \in \PathSet_{\DemandIdx}} \NodeCut_{\DemandIdx \PathIdx}.
\end{equation*}
We have
\begin{equation*}
\AggV[\DemandIdx](\MasterVar) 
= \FlowVolume_{\DemandIdx} \AggSubprobVar_{\DemandIdx} 
\le \FlowVolume_{\DemandIdx} \sum_{\NodeIdx \in \NodeCut_{\DemandIdx}} \MasterVar_{\NodeIdx}
\le \FlowVolume_{\DemandIdx} \sum_{\PathIdx \in \PathSet_{\DemandIdx}} \sum_{\NodeIdx \in \NodeCut_{\DemandIdx \PathIdx}} \MasterVar_{\NodeIdx}
= \FlowVolume_{\DemandIdx} \sum_{\PathIdx \in \PathSet_{\DemandIdx}} \DisaggSubprobVar{\DemandIdx}{\PathIdx}
= \DisaggV[\DemandIdx](\MasterVar).
\end{equation*}
\item 
The equality follows from the definitions.
\item 
We have seen that $\AggV[\DemandIdx](\MasterVar) = \DisaggV[\DemandIdx](\MasterVar)$, so we will show $\TightV[\DemandIdx](\MasterVar) = \AggV[\DemandIdx](\MasterVar)$.
Observe that
\begin{align*}
\mathrm{subgraph}(\AggV[\DemandIdx]) = 
\Bigg\{ (\MasterVar, \theta) \in [0, 1]^{|\NodeSet|} \times \mathbb{R} : \FlowVolume_{\DemandIdx} \sum_{\NodeIdx \in \NodeCut} \MasterVar_{\NodeIdx} \ge \theta, \NodeCut \in \AggNodeCutSet{\DemandIdx} \Bigg\}.
\end{align*}
Similar to Proposition~\ref{prop_integrality_of_set_covering_polytope}, we can show that if $\PathSet_{\DemandIdx}$ consists of a single, simple path, all the vertices of $\mathrm{subgraph}(\AggV[\DemandIdx])$ are $\{ (\MasterVar, \AggV[\DemandIdx](\MasterVar)) : \MasterVar \in \{0, 1\}^{|\NodeSet|} \}$.
Thus,
\begin{align*}
\mathrm{subgraph}(\AggV[\DemandIdx]) = 
\mathrm{conv} \{ (\MasterVar, \AggV[\DemandIdx](\MasterVar)) : \MasterVar \in \{0, 1\}^{|\NodeSet|} \} + \{ (0, r) : r \le 0 \}.
\end{align*}
Lemma~\ref{lemma_tightv_formula} implies that $\mathrm{subgraph}(\TightV[\DemandIdx])$ has the same set of vertexes and
\begin{align*}
\mathrm{subgraph}(\TightV[\DemandIdx]) = 
\mathrm{conv} \{ (\MasterVar, \AggV[\DemandIdx](\MasterVar)) : \MasterVar \in \{0, 1\}^{|\NodeSet|} \} + \{ (0, r) : r \le 0 \}.
\end{align*}
Thus, $\AggV[\DemandIdx](\MasterVar) = \TightV[\DemandIdx](\MasterVar)$ for any $\MasterVar \in [0, 1]^{|\NodeSet|}$.
\end{enumerate}
\end{proof}

Proposition~\ref{prop_comparison_of_v_on_general_point} implies that the LP relaxation of~\eqref{eq_disagg_extended_formulation} may yield a weaker (i.e., looser) upper bound than that of~\eqref{eq_agg_extended_formulation}.
A natural question, then, is how large the difference between these two bounds can be.
The following proposition shows that the relative gap between the bounds can be arbitrarily large.

\begin{proposition}
\label{prop_looseness_of_relaxations}
\begin{enumerate}
\item
For any $\gamma > 1$, there exists a deviation FRLP instance such that
\begin{equation*}
\gamma \cdot \hspace{-0.5cm}\underbrace{\max_{x \in \bar{X}} \sum_{q \in Q} v_q^{\text{agg}}(x)}
_
{\text{UB from LP relaxation of~\eqref{eq_agg_extended_formulation}}}
\le
\underbrace{\max_{x \in \bar{X}} \sum_{q \in Q} v_q^{\text{disagg}}(x).}
_
{\text{UB from LP relaxation of~\eqref{eq_disagg_extended_formulation}}}
\end{equation*}
\item
For any $\gamma > 1$, there exists a deviation FRLP instance such that
\begin{equation*}
\gamma \cdot \hspace{-0.5cm} \underbrace{\max_{\MasterVar \in \LprMasterFeasSet} \sum_{\DemandIdx \in \DemandSet} \TightV[\DemandIdx](\MasterVar)}
_
{\text{UB from LP relaxation with $\TightV$}}
\le
\underbrace{\max_{\MasterVar \in \LprMasterFeasSet} \sum_{\DemandIdx \in \DemandSet} \AggV[\DemandIdx](\MasterVar).}
_
{\text{UB from LP relaxation of~\eqref{eq_agg_extended_formulation}}}
\end{equation*}
\end{enumerate}
\end{proposition}

\begin{proof}
\begin{enumerate}
\item 
Fix $\gamma > 1$ and let $n$ be a positive integer such that
\begin{equation*}
\gamma < \frac{n + 1}{2}.
\end{equation*}
Consider the following graph $G = (\NodeSet, \EdgeSet)$, where $\NodeSet = \{1, 2, \ldots, 2n\}$ and 
\begin{equation*}
\EdgeSet = \{ (1, \NodeIdx), (\NodeIdx, n + 2) : \NodeIdx = 2, 3, \ldots, n + 1 \} \cup \{ (\NodeIdx, \NodeIdx + 1) : \NodeIdx = n + 2, n + 3, \ldots, 2n - 1 \}.
\end{equation*}
Suppose each edge is of length $\TravelRange$, where $\TravelRange$ is the travel range. The graph is shown in Figure~\ref{fig_proof_of_prop_looseness_of_relaxations_1}.

\begin{figure}[htbp]
\centering
\begin{tikzpicture}[
xscale=3,
point/.style={circle,inner sep=0,minimum size=1mm,fill=black},
hypograph/.style={ultra thick,draw=Rhodamine},
]

\coordinate (v1) at (0.0, 0.0);
\coordinate (v2) at (0.5, 1.5);
\coordinate (v3) at (0.5, 0.6);
\coordinate (vnp1) at (0.5, -1.5);
\coordinate (vnp2) at (1.0, 0.0);
\coordinate (vnp3) at (1.5, 0.0);
\coordinate (vnp4) at (2.0, 0.0);
\coordinate (vnp5) at (2.6, 0.0);

\node[point,label={below:1}] at (v1) {};
\node[point,label={below:2}] at (v2) {};
\node[point,label={below:3}] at (v3) {};
\node[point,label={below:$n + 1$}] at (vnp1) {};
\node at ($(v3)!0.5!(vnp1)$) {$\vdots$};
\node[point,label={[xshift=0.2cm]below:$n + 2$}] at (vnp2) {};
\node[point,label={below:$n + 3$}] at (vnp3) {};
\node[point,label={below:$n + 4$}] at (vnp4) {};
\node[point,label={below:$2 n$}] at (vnp5) {};
\node at ($(vnp4)!0.5!(vnp5)$) {$\cdots$};

\draw
(v1) -- node[pos=0.6,label={[label distance=-0.15cm]above:\small $\TravelRange$}]{} (v2)
(v1) -- node[pos=0.6,label={[label distance=-0.15cm]above:\small $\TravelRange$}]{} (v3)
(v1) -- node[pos=0.6,label={[label distance=-0.15cm]above:\small $\TravelRange$}]{} (vnp1)
(v2) -- node[pos=0.4,label={[label distance=-0.15cm]above:\small $\TravelRange$}]{} (vnp2)
(v3) -- node[pos=0.4,label={[label distance=-0.15cm]above:\small $\TravelRange$}]{} (vnp2)
(vnp1) -- node[pos=0.4,label={[label distance=-0.15cm]above:\small $\TravelRange$}]{} (vnp2)
(vnp2) -- node[above]{\small $\TravelRange$} (vnp3)
(vnp3) -- node[above]{\small $\TravelRange$} (vnp4)
(vnp4) -- +(0.15, 0.0)
(vnp5) -- +(-0.15, 0.0)
;

\end{tikzpicture}
\caption{Graph $G$, where $\TravelRange$ is the travel range}
\label{fig_proof_of_prop_looseness_of_relaxations_1}
\end{figure}

Let $\DemandSet = \{1\}$, where the origin and the destination of the demand are $1$ and $2 n$, respectively, and $\PathSet_1$ contains all the $(1, 2n)$-path of length $n d$:
\begin{equation*}
\PathSet_1 = \{r_{\NodeIdx} := (1, \NodeIdx + 1, n + 2, n + 3, \ldots, 2n) : \NodeIdx = 1, 2, \ldots, n \}.
\end{equation*}
Finally, suppose that
\begin{equation*}
\LprMasterFeasSet = \left\{ \MasterVar \in \mathbb{R}^{2n} : \sum_{\NodeIdx = 1}^{2 n} \MasterVar_{\NodeIdx} \le 2 \right\}.
\end{equation*}
On this instance, following the argument by \citet{CaparEtAl2013}, for $\NodeIdx = 1, 2, \ldots, n$, we have
\begin{equation*}
\DisaggNodeCutSet{1}{r_{\NodeIdx}} = \Big\{ \{ 1 \}, \{ j + 1 \}, \{ n + 2 \}, \{ n + 3\}, \ldots, \{ 2 n \} \Big\}.
\end{equation*}
That is, each path $r \in \PathSet_1$ is traversable if and only if all nodes in $r$ have charging stations.
Let
\begin{align*}
\MasterVar_{\NodeIdx}' &= \frac{1}{n}, \ \ \NodeIdx = 1, 2, \ldots, 2 n.
\end{align*}
Note that $\MasterVar' \in \LprMasterFeasSet$.
Since
\begin{align*}
\DisaggSubprobVar{1}{\PathIdx}' = \frac{1}{n}, \ \ \PathIdx \in \PathSet_1,
\end{align*}
is feasible for the LP in the RHS of~\eqref{eq_disagg_formulation} for $\MasterVar = \MasterVar'$,
\begin{equation*}
\max_{x \in \bar{X}} \sum_{q \in Q} v_q^{\text{disagg}}(x)
\ge v_1^{\text{disagg}}(x')
\ge \sum_{\PathIdx \in \PathSet_1} \FlowVolume_1 \DisaggSubprobVar{1}{\PathIdx}' = \FlowVolume_1.
\end{equation*}
Next, we will obtain a bound on $\AggV[1]$.
It is clear that
\begin{equation*}
\{ 1 \}, \{ 2, 3, \ldots, n + 1 \}, \{ n + 2 \}, \{ n + 3\}, \ldots, \{ 2 n \} \in \AggNodeCutSet{1}.
\end{equation*}
Therefore, for any $\MasterVar \in \LprMasterFeasSet$, we have
\begin{align*}
\AggV[1](\MasterVar)
&= \max_{\AggSubprobVar_1 \in [0, 1]} \left\{ \FlowVolume_1 \AggSubprobVar_1 : \sum_{\NodeIdx \in \NodeCut} \MasterVar_{\NodeIdx} \ge \AggSubprobVar_1, \ \ \forall \NodeCut \in \AggNodeCutSet{1} \right\} \\
&= \max_{\AggSubprobVar_1 \in [0, 1]} \left\{ \FlowVolume_1 \AggSubprobVar_1 : \MasterVar_1 \ge \AggSubprobVar_1, \sum_{\NodeIdx = 2}^{n + 1} \MasterVar_{\NodeIdx} \ge \AggSubprobVar_1, \MasterVar_{n + 2} \ge \AggSubprobVar_1, \MasterVar_{n + 3} \ge \AggSubprobVar_1, \ldots, \MasterVar_{2 n} \ge \AggSubprobVar_1 \right\}.
\end{align*}
Sum up the constraints to get
\begin{align*}
&\max_{\AggSubprobVar_1 \in [0, 1]} \left\{ \FlowVolume_1 \AggSubprobVar_1 : \MasterVar_1 \ge \AggSubprobVar_1, \sum_{\NodeIdx = 2}^{n + 1} \MasterVar_{\NodeIdx} \ge \AggSubprobVar_1, \MasterVar_{n + 2} \ge \AggSubprobVar_1, \MasterVar_{n + 3} \ge \AggSubprobVar_1, \ldots, \MasterVar_{2 n} \ge \AggSubprobVar_1 \right\} \\
&\qquad\le \max_{\AggSubprobVar_1 \in [0, 1]} \left\{ \FlowVolume_1 \AggSubprobVar_1 : \sum_{\NodeIdx = 1}^{2 n} \MasterVar_{\NodeIdx} \ge (n + 1) \AggSubprobVar_1 \right\}.
\end{align*}
For $\MasterVar \in \LprMasterFeasSet$, we have $\sum_{\NodeIdx = 1}^{2 n} \MasterVar_{\NodeIdx} \le 2$.
Therefore,
\begin{align*}
\max_{\AggSubprobVar_1 \in [0, 1]} \left\{ \FlowVolume_1 \AggSubprobVar_1 : \sum_{\NodeIdx = 1}^{2 n} \MasterVar_{\NodeIdx} \ge (n + 1) \AggSubprobVar_1 \right\}
\le \max_{\AggSubprobVar_1 \in [0, 1]} \left\{ \FlowVolume_1 \AggSubprobVar_1 : 2 \ge (n + 1) \AggSubprobVar_1 \right\}
\le \frac{2}{n + 1} \FlowVolume_1.
\end{align*}
By combining the above inequalities, we get
\begin{equation*}
\max_{x \in \bar{X}} \sum_{q \in Q} v_q^{\text{agg}}(x) \le \frac{2}{n + 1} \FlowVolume_1.
\end{equation*}
Therefore, we have
\begin{equation*}
\gamma \max_{x \in \bar{X}} \sum_{q \in Q} v_q^{\text{agg}}(x)
\le \frac{n + 1}{2} \max_{x \in \bar{X}} \sum_{q \in Q} v_q^{\text{agg}}(x)
\le \FlowVolume_1 \le \max_{x \in \bar{X}} \sum_{q \in Q} v_q^{\text{disagg}}(x). 
\end{equation*}
\item 
Fix $\gamma > 1$ and let $n \ge \max\{6 \gamma, 2\}$ be an integer.
Consider a graph $G = (\NodeSet, \EdgeSet)$, where $\NodeSet = \{1, 2, \ldots, 2 n + 4\}$ and
\begin{equation*}
\EdgeSet = \{(1, 2), (2, 3), (3, 2 n + 3), (2 n + 3, 2 n + 4)\} \cup \{ (i, j) : 3 \le i, j \le 2 n + 2, i \not= j \}.
\end{equation*}
The edges $1$-$2$ and $(2 n + 3)$-$(2 n + 4)$ are of length $\TravelRange - \delta$, edges $2$-$3$ and $3$-$(2 n + 1)$ are of length $2 \delta$ and all the other edges are of length $(\TravelRange - \delta) / n$, where $0 < \delta < 1 / n^2$ and $\TravelRange$ is the travel range.
The graph is illustrated in Figure~\ref{fig_proof_of_prop_looseness_of_relaxations_2_general_n}.

\begin{figure}[htbp]
\centering
\begin{tikzpicture}[
xscale=3,
point/.style={circle,inner sep=0,minimum size=1mm,fill=black},
hypograph/.style={ultra thick,draw=Rhodamine},
]

\coordinate (v1) at (-1.5, 0.0);
\coordinate (v2) at (-0.5, 0.0);
\coordinate (v3) at ( 0.0, 0.0);
\coordinate (v4) at (-0.6495, 1.5);
\coordinate (v5) at (-0.6495, 3.75);
\coordinate (v6) at ( 0.0, 5.4);
\coordinate (v2np2) at ( 0.6495, 1.5);
\coordinate (v2np1) at ( 0.6495, 3.75);
\coordinate (v2np3) at ( 0.5, 0.0);
\coordinate (v2np4) at ( 1.5, 0.0);


\node[point,label={below:1}] at (v1) {};
\node[point,label={below:2}] at (v2) {};
\node[point,label={below:3}] at (v3) {};
\node[point,label={left:4}] at (v4) {};
\node[point,label={left:5}] at (v5) {};
\node[point,label={right:$2 n + 1$}] at (v2np1) {};
\node[point,label={right:$2 n + 2$}] at (v2np2) {};
\node[point,label={below:$2 n + 3$}] at (v2np3) {};
\node[point,label={below:$2 n + 4$}] at (v2np4) {};



\draw
(v1) -- node[below]{\small $\TravelRange - \delta$} (v2)
(v2) -- node[below]{\small $2\delta$} (v3)
(v3) -- node[pos=0.4,label={[label distance=0.1cm]left:\small $(\TravelRange - \delta) / n$}] {} (v4)
(v3) -- node[pos=0.4,label={[label distance=0.1cm]right:\small $(\TravelRange - \delta) / n$}] {} (v2np2)
(v3) -- node[pos=1.5,rotate=-65] {$\cdots$}($(v3)!0.15!(v5)$)
(v3) -- node[pos=1.5,rotate=65] {$\cdots$} ($(v3)!0.15!(v2np1)$)
(v3) -- node[pos=1.5,rotate=90] {$\cdots$}($(v3)!0.12!(v6)$)
(v4) -- node[pos=0.5,label={[label distance=0.1cm]left:\small $(\TravelRange - \delta) / n$}] {} (v5)
(v2np1) -- node[pos=0.5,label={[label distance=0.1cm]right:\small $(\TravelRange - \delta) / n$}] {} (v2np2)
(v4) -- node[pos=1.7,rotate=65] {$\cdots$} ($(v4)!0.1!(v6)$)
(v2np2) -- node[pos=1.7,rotate=-65] {$\cdots$}($(v2np2)!0.1!(v6)$)
(v4) -- node[pos=1.6,rotate=35] {$\cdots$} ($(v4)!0.1!(v2np1)$)
(v2np2) -- node[pos=1.6,rotate=-35] {$\cdots$} ($(v2np2)!0.1!(v5)$)
(v4) -- node[pos=1.6,rotate=0] {$\cdots$} ($(v4)!0.1!(v2np2)$)
(v2np2) -- node[pos=1.6,rotate=0] {$\cdots$} ($(v2np2)!0.1!(v4)$)
(v5) -- node[pos=1.5,rotate=-65] {$\cdots$} ($(v5)!0.15!(v3)$)
(v2np1) -- node[pos=1.5,rotate=65] {$\cdots$} ($(v2np1)!0.15!(v3)$)
(v5) -- node[pos=1.6,rotate=35] {$\cdots$} ($(v5)!0.18!(v6)$)
(v2np1) -- node[pos=1.6,rotate=-35] {$\cdots$}($(v2np1)!0.18!(v6)$)
(v5) -- node[pos=1.7,rotate=0] {$\cdots$} ($(v5)!0.1!(v2np1)$)
(v2np1) -- node[pos=1.7,rotate=0] {$\cdots$} ($(v2np1)!0.1!(v5)$)
(v5) -- node[pos=1.7,rotate=-35] {$\cdots$} ($(v5)!0.1!(v2np2)$)
(v2np1) -- node[pos=1.7,rotate=35] {$\cdots$} ($(v2np1)!0.1!(v4)$)
(v3) -- node[below]{\small $2\delta$} (v2np3)
(v2np3) -- node[below]{\small $\TravelRange - \delta$} (v2np4)
;

\end{tikzpicture}
\caption{Graph $G$, where $0 < \delta < 1 / n^2$ and $\TravelRange$ is the travel range}
\label{fig_proof_of_prop_looseness_of_relaxations_2_general_n}
\end{figure}

Suppose that $\DemandSet = \{ 1 \}$, where the origin and the destination of the demand are 1 and $2 n + 4$, respectively.
Furthermore, assume that $\PathSet_1$ contains a single path $r = (1, 2, r', 2 n + 3, 2 n + 4)$, where $r'$ is a path that starts and ends at $3$ and contains all the permutations of $\{3, 4, \ldots, 2 n + 2\}$.
For example, in case of $n = 3$, we may have
\begin{equation*}
r = (1, 2, \ \
3, 4, 5, 6, 7, 8, \ \
3, 4, 5, 6, 8, 7, \ \
\ldots, \ \
8, 7, 6, 5, 4, 3, \ \
9, 10).
\end{equation*}
We have
\begin{equation*}
\AggNodeCutSet{1} = \Big\{
\{ 1 \},
\{ 2 \},
\{ 2 n + 3 \},
\{ 2 n + 4 \}
\Big\} \cup
\Big\{
S : S \subset \{3, 4, \ldots, 2 n + 2\}, |S| = n
\Big\}.
\end{equation*}
Note that the demand is served if and only if
\begin{equation}
\MasterVar_1
= \MasterVar_2
= \MasterVar_{n + 3}
= \MasterVar_{n + 4}
= 1
\ \MyAnd \
\sum_{\NodeIdx = 3}^{2 n + 2} \MasterVar_{\NodeIdx} \ge n + 1.
\label{v_disagg_loose_example_demand_condition}
\end{equation}
Now, let us consider $\LprMasterFeasSet$ given by a simple budget constraint
\begin{equation*}
\LprMasterFeasSet = \left\{ \MasterVar \in \mathbb{R}^{2n} : \sum_{\NodeIdx = 1}^{2 n + 4} \MasterVar_{\NodeIdx} \le 6 \right\}.
\end{equation*}
It is straightforward to see that  $\AggV[1](x') = \FlowVolume_1$ for
\begin{equation*}
\MasterVar_{\NodeIdx}' = \begin{cases}
1, & \NodeIdx \in \{ 1, 2, 2 n + 3, 2 n + 4 \}, \\
1 / n, & \text{otherwise}.
\end{cases}
\end{equation*}
We will obtain a bound on $\TightV[1](\MasterVar)$ for $\MasterVar \in \LprMasterFeasSet$.
In the proof of Lemma~\ref{lemma_tightv_formula}, we showed that
\begin{align*}
\TightV[1](\MasterVar)
&= \max_{\alpha} \left\{ \sum_{i = 1}^{2^{|\NodeSet|}} \alpha_i \AggV[1](\beta_i) 
:
\sum_{i = 1}^{2^{|\NodeSet|}} \alpha_i \beta_i = \MasterVar,
\sum_{i = 1}^{2^{|\NodeSet|}} \alpha_i = 1, \alpha \ge 0
\right\},
\end{align*}
where $\{0, 1\}^{|\NodeSet|} = \{ \beta_i : i = 1, 2, \ldots, 2^{|\NodeSet|} \}$.
Let $I_{1}, I_{2}$ be a partition of $\{ 1, 2, \ldots, 2^{|\NodeSet|} \}$ such that
\begin{align}
\| \beta_i \|_1 \le n - 1, \ \ \forall i \in I_{1}, \qquad
\| \beta_i \|_1 \ge n, \ \ \forall i \in I_{2}.
\label{v_disagg_loose_example_definition_of_partition}
\end{align}
It follows from~\eqref{v_disagg_loose_example_demand_condition} that
\begin{equation*}
\AggV[1](\beta_i) = \begin{cases}
0, & i \in I_{1}, \\
\FlowVolume_1, & i \in I_{2},
\end{cases}
\end{equation*}
and thus, for $\MasterVar \in \LprMasterFeasSet$,
\begin{align*}
\TightV[1](\MasterVar)
&= \max_{\alpha} \left\{ \FlowVolume_1 \sum_{i \in I_{2}} \alpha_i
:
\sum_{i = 1}^{2^{|\NodeSet|}} \alpha_i \beta_i = \MasterVar,
\sum_{i = 1}^{2^{|\NodeSet|}} \alpha_i = 1, \alpha \ge 0
\right\}.
\end{align*}
Now, we will obtain a bound by constructing relaxations.
First, observe that
\begin{align*}
\TightV[1](\MasterVar)
&= \max_{\alpha} \left\{ \FlowVolume_1 \sum_{i \in I_{2}} \alpha_i
:
\sum_{i = 1}^{2^{|\NodeSet|}} \alpha_i \beta_i = \MasterVar,
\sum_{i = 1}^{2^{|\NodeSet|}} \alpha_i = 1, \alpha \ge 0
\right\} \\
&\le \max_{\alpha} \left\{ \FlowVolume_1 \sum_{i \in I_{2}} \alpha_i
:
\left\| \sum_{i = 1}^{2^{|\NodeSet|}} \alpha_i \beta_i \right\|_1 = \| \MasterVar \|_1,
\sum_{i = 1}^{2^{|\NodeSet|}} \alpha_i = 1, \alpha \ge 0
\right\} \\
&= \max_{\alpha} \left\{ \FlowVolume_1 \sum_{i \in I_{2}} \alpha_i
:
\sum_{i = 1}^{2^{|\NodeSet|}} \alpha_i \left\| \beta_i \right\|_1 = \| \MasterVar \|_1,
\sum_{i = 1}^{2^{|\NodeSet|}} \alpha_i = 1, \alpha \ge 0
\right\} \\
&= \max_{\alpha} \left\{ \FlowVolume_1 \sum_{i \in I_{2}} \alpha_i
:
\sum_{i \in I_{1}} \alpha_i \left\| \beta_i \right\|_1 + \sum_{i \in I_{2}} \alpha_i \left\| \beta_i \right\|_1 = \| \MasterVar \|_1,
\sum_{i = 1}^{2^{|\NodeSet|}} \alpha_i = 1, \alpha \ge 0
\right\} \\
&\le \max_{\alpha} \left\{ \FlowVolume_1 \sum_{i \in I_{2}} \alpha_i
:
\sum_{i \in I_{2}} \alpha_i \left\| \beta_i \right\|_1 \le \| \MasterVar \|_1,
\sum_{i = 1}^{2^{|\NodeSet|}} \alpha_i = 1, \alpha \ge 0
\right\}.
\end{align*}
Since $\| \beta_i \|_1 \ge n$ for $i \in I_{2}$ and $\| \MasterVar \|_1 \le 6$ for $\MasterVar \in \LprMasterFeasSet$, we have
\begin{align*}
& \max_{\alpha} \left\{ \FlowVolume_1 \sum_{i \in I_{2}} \alpha_i
:
\sum_{i \in I_{2}} \alpha_i \left\| \beta_i \right\|_1 \le \| \MasterVar \|_1,
\sum_{i = 1}^{2^{|\NodeSet|}} \alpha_i = 1, \alpha \ge 0
\right\} \\
&\qquad\le \max_{\alpha} \left\{ \FlowVolume_1 \sum_{i \in I_{2}} \alpha_i
:
n \sum_{i \in I_{2}} \alpha_i \le 6,
\sum_{i = 1}^{2^{|\NodeSet|}} \alpha_i = 1, \alpha \ge 0
\right\}.
\end{align*}
Therefore, it follows that for $\MasterVar \in \LprMasterFeasSet$
\begin{equation*}
\TightV[1](x) \le \frac{6}{n} \FlowVolume_1.
\end{equation*}
Thus,
\begin{equation*}
\gamma \max_{\MasterVar \in \LprMasterFeasSet} \sum_{\DemandIdx \in \DemandSet} \TightV[\DemandIdx](\MasterVar)
\le
\frac{n}{6} \max_{\MasterVar \in \LprMasterFeasSet} \sum_{\DemandIdx \in \DemandSet} \TightV[\DemandIdx](\MasterVar)
\le
\FlowVolume_1
=
\max_{\MasterVar \in \LprMasterFeasSet} \sum_{\DemandIdx \in \DemandSet} \AggV[\DemandIdx](\MasterVar).
\end{equation*}
\end{enumerate}
\end{proof}

\section{Cyclic FRLP}\label{sec_cyclic_frlp}

\subsection{Formulations}\label{subsec_formulations_of_the_cyclic_frlp}

In Sections~\ref{sec_formulations_of_frlp} and~\ref{sec_strength_of_formulations}, we studied the FRLP under the standard setting considered in the literature (e.g., \citet{ArslanEtAl2019} and \citet{GopfertAndBock2019}), where it is assumed that drivers use the same path in both directions, from the origin to the destination and back.
An OD pair is thus considered served if drivers can traverse the path repeatedly. 
Furthermore, this symmetric routing requires all roads to be undirected.
However, in many practical situations, it may be natural to consider the possibility that drivers use different paths for their outbound and return trips. 
In this setting, an OD path is considered served if one can repeat a cycle from the origin to the destination and back, without running out of battery.
We call this extension of the FRLP, which allows such flexibility, as the \emph{cyclic FRLP}; 
see Figure~\ref{fig_cyclic_frlp_illustration} for an illustration.
In contrast, we refer to the deviation FRLP (including the classic FRLP), where paths are symmetric, as the \emph{original FRLP}.
In this section, we develop a formulation and solution method for the cyclic FRLP.

\begin{figure}[htbp]
\centering
\begin{tikzpicture}[
xscale=1,
point/.style={circle,inner sep=0,minimum size=1mm,fill=black},
nodepadding/.style={circle,inner sep=0,minimum size=6mm},
unselected/.style={draw=black, densely dashed},
selected/.style={draw=black},
hypograph/.style={ultra thick,draw=Rhodamine},
]
\coordinate (v0) at (0.0, 0.0);
\coordinate (v1) at (5.0, 2.0);

\node[nodepadding] (n0) at (v0) {};
\node[nodepadding] (n1) at (v1) {};

\node[point,label={below:1}] at (v0) {};
\node[point,label={below right:2}] at (v1) {};

\draw[unselected] (n0) to [out=70,in=150] node [below] {} (n1);

\draw[unselected] (n0) to node [below] {} (n1);
\draw[unselected] (n0) to [bend right=45] node [below] {} (n1);

\begin{scope}[shift={(-0.08,0.16)}]
\coordinate (v0) at (0.0, 0.0);
\coordinate (v1) at (5.0, 2.0);

\node[nodepadding] (n0) at (v0) {};
\node[nodepadding] (n1) at (v1) {};

\draw[selected, -{Latex[length=2mm]}] (n0) to [out=70,in=150] (n1);
\end{scope}

\begin{scope}[shift={(-0.08,0.16)}]
\coordinate (v0) at (0.0, 0.0);
\coordinate (v1) at (5.0, 2.0);

\node[nodepadding] (n0) at (v0) {};
\node[nodepadding] (n1) at (v1) {};

\draw[selected, {Latex[length=2mm]}-] (n0) to (n1);
\end{scope}

\end{tikzpicture}
\caption{OD pair with origin and destination at nodes $1$ and $2$, respectively; dashed lines indicate possible OD paths, and solid directed lines show a feasible route for the cyclic FRLP. In the cyclic FRLP, drivers may choose different paths for their outbound and inbound trips. An OD pair is served if there is a cycle that can be repeated without running out of battery (recall Figure~\ref{fig_original_frlp_illustration} for comparison).}
\label{fig_cyclic_frlp_illustration}
\end{figure}

For example, consider a case where a driver deviates from the shortest path on the outbound trip to visit a charging station.
In the original FRLP, it is assumed that the driver must also follow the same path on the return trip and visit the same charging station again.
However, if the battery level is sufficiently high at the destination, the driver may choose a shorter path for the return trip, deferring recharging to the next trip.
The following example illustrates such a scenario, highlighting the difference between the original and cyclic FRLP.

\begin{example}

Consider the undirected graph shown in Figure~\ref{fig_frlp_original_vs_cyclic}, with a single demand \(\DemandIdx_{1}\) from node 1 to node 2 (i.e., \(\DemandSet = \{ \DemandIdx_{1} \}\)).  
Assume that the travel range is \(\TravelRange\).

\begin{figure}[htbp]
\centering
\begin{tikzpicture}[
point/.style={circle,inner sep=0,minimum size=1mm,fill=black},
hypograph/.style={ultra thick,draw=Rhodamine},
]
\coordinate (v0) at (0.0, 0.0);
\coordinate (v1) at (4.0, 0.0);
\coordinate (v2) at (2.0, 1.5);
\coordinate (v3) at (2.0,-1.5);

\node[point,label={below:1}] at (v0) {};
\node[point,label={below:2}] at (v1) {};
\node[point,label={above:3}] at (v2) {};
\node[point,label={below:4}] at (v3) {};

\draw 
    (v0) -- node[below]{\small $\TravelRange / 3$} (v1)
    (v0) -- node[above left]{\small $\TravelRange / 4$} (v2)
    (v1) -- node[above right]{\small $\TravelRange / 4$} (v2)
    (v0) -- node[below left]{\small $\TravelRange / 3$} (v3)
    (v1) -- node[below right]{\small $\TravelRange / 3$} (v3)
    ;


\end{tikzpicture}
\caption{Graph $G$, where $\TravelRange$ is the travel range}
\label{fig_frlp_original_vs_cyclic}
\end{figure}

\begin{table}[htbp]
\begin{minipage}[m]{0.46\linewidth}
\centering
\captionof{table}{Paths from node 1 to node 2 within a 50\% deviation from the shortest path}
\label{tab_frlp_original_vs_cyclic_route_symmetric_frlp_path_list}
\begin{tabular*}{\linewidth}{@{\extracolsep{\fill}}lrr}
\toprule
Path & Length & Deviation \\
\midrule
$(1, 2)$ & $(1/3) \TravelRange$ & 0\% \\
$(1, 3, 2)$ & $(1/2) \TravelRange$ & 50\% \\
\bottomrule
\end{tabular*}
\end{minipage}%
\hfill
\begin{minipage}[m]{0.46\linewidth}
\centering
\captionof{table}{Cycles from node 1 to 2 and back, within a 50\% deviation from the shortest cycle}
\label{tab_frlp_original_vs_cyclic_route_asymmetric_frlp_cycle_list}
\begin{tabular*}{\linewidth}{@{\extracolsep{\fill}}lrr}
\toprule
Cycle & Length & Deviation \\
\midrule
$(1, 2, 1)$ & $(2/3) \TravelRange$ & 0\% \\
$(1, 3, 2, 3, 1)$ & $\TravelRange$ & 50\% \\
$(1, 2, 3, 1)$ & $(5/6) \TravelRange$ & 25\% \\
$(1, 2, 4, 1)$ & $\TravelRange$ & 50\% \\
\bottomrule
\end{tabular*}
\end{minipage}
\end{table}

In the original FRLP, if drivers are willing to take any path that is at most 50\% longer than the shortest one, there are two possible paths, as shown in Table~\ref{tab_frlp_original_vs_cyclic_route_symmetric_frlp_path_list}.  
Note that path $(1, 4, 2)$ has length \(2\TravelRange\), which exceeds the 50\% deviation limit.

Now consider the cyclic FRLP.
If drivers are willing to take any cycle that is at most 50\% longer than the shortest round-trip cycle, there are four possible cycles, as listed in Table~\ref{tab_frlp_original_vs_cyclic_route_asymmetric_frlp_cycle_list} (excluding symmetric duplicates).
The first two cycles correspond to the paths considered in the original FRLP.

Suppose that a charging station is available only at node 4.  
Under the original FRLP, demand \(\DemandIdx_{1}\) is considered unserved, as no feasible path passes through node 4.  
In contrast, under the cyclic FRLP, demand \(\DemandIdx_{1}\) is considered served, since the cycle $(1, 2, 4, 1)$ is feasible and can be repeated without violating the travel range constraint.
\end{example}

In the original FRLP, \citet{CaparEtAl2013} provided an argument that for a given path, one can construct a family of node sets $\NodeCutSet \subset \PowerSet{\NodeSet}$ such that the path is traversable repeatedly if and only if $\sum_{\NodeIdx \in \NodeCut} \MasterVar_{\NodeIdx} \ge 1$ for all $\NodeCut \in \NodeCutSet$.
Thus, for each $\DemandIdx \in \DemandSet$ and $\PathIdx \in \PathSet_{\DemandIdx}$, one can construct either $\AggNodeCutSet{\PathIdx}$ or $\DisaggNodeCutSet{\DemandIdx}{\DemandIdx}$ to encode the condition for $\PathIdx$ to be traversable, or for $\DemandIdx$ to be considered served, where $\PathSet_{\DemandIdx}$ denotes the set of allowable paths for demand $\DemandIdx$.

It is straightforward to extend the argument to the cyclic FRLP.
Given a cycle, one can construct a corresponding set $\NodeCutSet \subset \PowerSet{\NodeSet}$ such that the cycle is traversable repeatedly if and only if $\sum_{\NodeIdx \in \NodeCut} \MasterVar_{\NodeIdx} \ge 1$ for all $\NodeCut \in \NodeCutSet$.
In the cyclic FRLP, each demand $\DemandIdx \in \DemandSet$ is associated with a set of cycles that drivers of $\DemandIdx$ are willing to use.
We also use $\CycleSet_{\DemandIdx}$ to denote this set of cycles.
Then, for each $\DemandIdx \in \DemandSet$ and $\CycleIdx \in \CycleSet_{\DemandIdx}$, one can construct $\DisaggNodeCutSet{\DemandIdx}{\CycleIdx}$, or alternatively $\AggNodeCutSet{\DemandIdx}$, and and use them to formulate~\eqref{eq_disagg_extended_formulation} or~\eqref{eq_agg_extended_formulation}.

Observe that the same formulations~\eqref{eq_disagg_extended_formulation} and~\eqref{eq_agg_extended_formulation} can be used for both the original and cyclic FRLP.
Furthermore, the original FRLP is a special case of the cyclic FRLP in which all cycles are ``symmetric'' (drivers use the same path in both directions).
Thus, many of the results established in Section~\ref{sec_strength_of_formulations} extend naturally to the cyclic FRLP.
To be more specific, the bound given by the LP relaxation of formulation~\eqref{eq_agg_extended_formulation} is tighter than that of formulation~\eqref{eq_disagg_extended_formulation} (Proposition~\ref{prop_comparison_of_v_on_general_point} (i)), and the difference between them can be arbitrarily large (Proposition~\ref{prop_looseness_of_relaxations}).\footnote{We note that assertions (ii) and (iii) in Proposition~\ref{prop_comparison_of_v_on_general_point} require drivers to use a single, simple path between the origin and the destination, which does not generalize to a single, simple cycle in case of the cyclic FRLP.}
Therefore, we focus on formulation~\eqref{eq_agg_extended_formulation} in the remainder of this section.

\subsection{Solution Method}\label{subsec_solution_method}

\citet{ArslanEtAl2019} and \citet{GopfertAndBock2019} solve formulation~\eqref{eq_agg_extended_formulation} for the original FRLP using branch-and-cut algorithms.
In this section, we demonstrate how to adapt the method proposed by \citet{ArslanEtAl2019} to solve the cyclic FRLP using the same formulation.
Among the variants proposed in their work, we focus on the one that yields the best performance in their experiments.

In a branch-and-cut algorithm, formulation~\eqref{eq_agg_extended_formulation} is initially considered without constraint~\eqref{eq_agg_extended_formulation_constraint}, and a standard branch-and-bound algorithm is applied.
Each time the branch-and-bound algorithm finds a candidate integer solution $(x', y')$, the algorithm checks whether any constraint~\eqref{eq_agg_extended_formulation_constraint} is violated.
Specifically, it searches for  $\DemandIdx \in \DemandSet$ and $\NodeCut \in \AggNodeCutSet{\DemandIdx}$ such that the corresponding constraint is not satisfied at $(x', y')$.
If a violated constraint is found, it is added to the model, and the current solution is discarded.
Otherwise, the candidate solution is feasible for the full problem.

While the scheme of the branch-and-cut algorithm outlined above is quite generic, the task of finding violated constraints for a given integer solution is typically problem-specific, as it must be done efficiently.
We now describe this separation procedure. 
Let $(x', y')$ be a candidate integer solution for which we aim to identify a violated constraint~\eqref{eq_agg_extended_formulation_constraint} or its feasibility.
Fix $\DemandIdx \in \DemandSet$ such that $y_{\DemandIdx}' = 1$
(if $y_{\DemandIdx}' = 0$, constraint~\eqref{eq_agg_extended_formulation_constraint} is satisfied for all $\NodeCut \in \AggNodeCutSet{\DemandIdx}$).
Define $\NodeCut' = \{ \NodeIdx : x_{\NodeIdx}' = 0 \}$, the set of nodes where no charging station is installed in the candidate solution.
We first check if there exists a cycle in $\CycleSet_{\DemandIdx}$ traversable repeatedly, assuming we only install charging stations at nodes in $\NodeSet \setminus \NodeCut'$.
If such a cycle exists, then the constraint is satisfied for any $\NodeCut \in \AggNodeCutSet{\DemandIdx}$.
Otherwise, the constraint with $\NodeCut = \NodeCut'$ is violated.

In light of Proposition~\ref{prop_minimal_node_set}, it is desirable to find a minimal node set that violates constraint~\eqref{eq_agg_extended_formulation_constraint}.
To this end, we iteratively attempt to remove nodes from $\NodeCut'$.
For each $\NodeIdx \in \NodeCut'$, we check whether there exists a traversable cycle in $\CycleSet_{\DemandIdx}$ using the modified node set $\NodeCut' \setminus \{ \NodeIdx \}$ (i.e., we search for a traversable cycle assuming we install charging stations at nodes in $\NodeSet \setminus (\NodeCut' \setminus \{ \NodeIdx \})$).
If no such cycle exists, we can replace $\NodeCut'$ with $\NodeCut' \setminus \{ \NodeIdx \}$.
The final (reduced) $\NodeCut'$ is then added to the set of violated constraints.
The pseudocode is provided in Algorithm~\ref{algo_separation}.

\begin{algorithm}
\caption{Separation of violated constraints}\label{algo_separation}
\renewcommand{\algorithmicrequire}{\textbf{Input:}}
\begin{algorithmic}[1]
\Require Set of demands $Q$, set of cycles $\CycleSet$, candidate integer solution $(x', y')$
\Ensure Set $\mathcal{O}$ such that for each $(\DemandIdx, \NodeCut) \in \mathcal{O}$ constraint~\eqref{eq_agg_extended_formulation_constraint} is violated
\State $\mathcal{O} \gets \emptyset$
\For{$\DemandIdx \in \DemandSet$ such that $y_{\DemandIdx} = 1$}
\State $\NodeCut' \gets \{ \NodeIdx : x_{\NodeIdx}' = 0 \}$
\If{there exists a cycle in $\CycleSet_{\DemandIdx}$ traversable using nodes in $\NodeSet \setminus \NodeCut'$}
\State \textbf{continue}
\EndIf
\For{$\NodeIdx \in \NodeCut'$}
\If{No cycle in $\CycleSet_{\DemandIdx}$ is traversable using nodes in $\NodeSet \setminus (\NodeCut' \setminus \{ \NodeIdx \})$}
\State $\NodeCut' \gets \NodeCut' \setminus \{ \NodeIdx \}$
\EndIf
\EndFor
\State $\mathcal{O} \gets \mathcal{O} \cup \{ (\DemandIdx, \NodeCut') \}$
\EndFor
\end{algorithmic}
\end{algorithm}

The branch-and-cut approach and the separation procedure described above follow those of \citet{ArslanEtAl2019}. 
However, in the deviation FRLP, identifying a violated constraint corresponds to finding a traversable path, whereas in the cyclic FRLP, it involves finding a traversable cycle.
The symmetric path assumption in the deviation FRLP (i.e., that drivers must use the same path for both outbound and inbound trips) considerably simplifies the search for a traversable path.
It can be shown that a path is repeatedly traversable if and only if a driver can leave the origin with a half-charged battery and reach the destination with at least 50\% of the battery remaining.
This condition reduces the search for a feasible path to a shortest path problem on a transformed network (see \citet{arslan2014impacts} for details).
In contrast, it is more involved to find a traversable cycle from the origin to the destination and back.
For instance, it is not known in advance how much energy can be assumed at the origin or how much must remain upon arrival at the destination. Hence, we propose a new procedure to find a traversable cycle.

In the cyclic FRLP, the existence of a traversable cycle in $\CycleSet_{\DemandIdx}$ can be tested using a labeling algorithm (see \citet{irnich2005shortest}).
For example, suppose $\CycleSet_{\DemandIdx}$ includes all cycles of length at most $\LengthLimit$ that start at the origin, visit the destination, and return to the origin.
The graph is first modified by adding a sink node $\SinkNode$ connected from the origin via a zero-length edge.
Each label has the form $(\LabelCharged, \LabelIsOutbound, \LabelDistFromStart, \LabelDistFromCharger, \LabelChargeAtStart)$, where 
\begin{itemize}
\item
$\LabelCharged$ is 1 if and only if we have visited at least one charging station and 0 otherwise, 
\item 
$\LabelIsOutbound$ is 1 if and only if we have visited the destination and 0 otherwise, 
\item 
$\LabelDistFromStart$ is the total distance traveled from the origin, 
\item 
$\LabelDistFromCharger$ is the distance since the last charging station (equal to $\LabelDistFromStart$ if $\LabelCharged = 0$),
\item 
$\LabelChargeAtStart$ is the amount of charge (in distance units) needed at the origin before completing the cycle (equal to $\infty$ if $\LabelCharged = 0$).
\end{itemize}
We initialize the origin with a label $(0, 0, 0, 0, \Nan)$ if there are no charging stations at the origin or otherwise $(1, 0, 0, 0, 0)$.
A label $(\LabelCharged, \LabelIsOutbound, \LabelDistFromStart, \LabelDistFromCharger, \LabelChargeAtStart)$ at node $\NodeIdx_{1}$ can be extended to node $\NodeIdx_{2} \not= \SinkNode$ if and only if $\LabelDistFromStart + \EdgeLength_{\NodeIdx_{1} \NodeIdx_{2}} \le \LengthLimit$ and $\LabelDistFromCharger + \EdgeLength_{\NodeIdx_{1} \NodeIdx_{2}} \le \TravelRange$ where $\EdgeLength_{\NodeIdx_{1} \NodeIdx_{2}}$ is the length of the edge from node $\NodeIdx_{1}$ to node $\NodeIdx_{2}$ and $\TravelRange$ is the vehicle's travel range.
Similarly, a label $(\LabelCharged, \LabelIsOutbound, \LabelDistFromStart, \LabelDistFromCharger, \LabelChargeAtStart)$ at the origin can be extended to the sink $\SinkNode$ if and only if $\LabelDistFromCharger + \LabelChargeAtStart \le \TravelRange$.
The label at node $\NodeIdx_{2}$ resulting from an extension is given by the resource extension function $\ResourceExtensionFunction$.
It is defined by $\ResourceExtensionFunction: (\LabelCharged, \LabelIsOutbound, \LabelDistFromStart, \LabelDistFromCharger, \LabelChargeAtStart) \mapsto (\LabelCharged', \LabelIsOutbound', \LabelDistFromStart', \LabelDistFromCharger', \LabelChargeAtStart')$, where
\begin{align*}
\LabelCharged' &= \begin{cases}
0, & \LabelCharged = 0 \ \MyAnd \ x_{\NodeIdx_{2}} = 0 \ (\text{i.e., there is no charging stations at $\NodeIdx_{2}$}), \\
1, & \text{otherwise,}
\end{cases}
\\
\LabelIsOutbound' &= \begin{cases}
0, & \LabelIsOutbound = 0 \ \MyAnd \ \NodeIdx_{2} \text{ is not the destination}, \\
1, & \text{otherwise,}
\end{cases}
\\
\LabelDistFromStart' &= \LabelDistFromStart + \EdgeLength_{\NodeIdx_{1} \NodeIdx_{2}},
\\
\LabelDistFromCharger' &= \begin{cases}
\LabelDistFromCharger + \EdgeLength_{\NodeIdx_{1} \NodeIdx_{2}}, & x_{\NodeIdx_{2}} = 0, \\
0, & \text{otherwise,}
\end{cases}
\\
\LabelChargeAtStart' &= \begin{cases}
\LabelDistFromStart + \EdgeLength_{\NodeIdx_{1} \NodeIdx_{2}}, & \LabelCharged = 0 \ \MyAnd \ x_{\NodeIdx_{2}} = 1, \\
\LabelChargeAtStart, & \text{otherwise.}
\end{cases}
\end{align*}

To determine the order in which labels are extended, we use a scoring rule. 
Let $\DistanceFunction(\NodeIdx_{1}, \NodeIdx_{2})$ denote the shortest distance from node $\NodeIdx_{1}$ to node $\NodeIdx_{2}$, and let $o$ and $d$ represent the origin and destination, respectively.
Given a label $(\LabelCharged, \LabelIsOutbound, \LabelDistFromStart, \LabelDistFromCharger, \LabelChargeAtStart)$ attached to a node $\NodeIdx$, its score is computed as $\DistanceFunction(\NodeIdx, d) + \DistanceFunction(d, o)$ if $\LabelIsOutbound = 0$, and as $\DistanceFunction(\NodeIdx, o)$ otherwise.
In other words, the score represents the length of the shortest cycle that could be formed by extending the label, ignoring travel range or distance limits. 
The label with the smallest score is selected for extension.

A traversable cycle in $\CycleSet_{\DemandIdx}$ exists if and only if the sink node $\SinkNode$ receives a label during the labeling process. When $\SinkNode$ is labeled, the condition $\LabelDistFromCharger + \LabelChargeAtStart \le \TravelRange$ ensures that the sum of the distance from the first charger after departure and the distance from the last charger before returning is at most the vehicle's travel range, thereby guaranteeing that the cycle can be repeated. Conversely, if a traversable cycle exists, there must be a corresponding sequence of labels that will be identified during the execution of the labeling algorithm.

Figure~\ref{fig_frlp_labelling} illustrates an execution of the labeling algorithm on a simple 3-node transportation network.
In this example, the sink $\SinkNode$ is labeled after five steps, confirming the existence of a cycle that satisfies both the travel range and total length constraints.

\begin{figure}[htbp]
\centering
\begin{tikzpicture}[
point/.style={circle,inner sep=0,minimum size=1.5mm,draw=black},
fullpoint/.style={circle,inner sep=0,minimum size=1.5mm,fill=black},
MyTitle/.style={font={\small\bfseries},anchor=west},
MyLabel/.style={rotate=45,font={\small},align=right,anchor=east},
MyLabelU/.style={rotate=45,font={\small},align=left,anchor=west},
hypograph/.style={ultra thick,draw=Rhodamine},
]








\begin{scope}[shift={(0,0)}]

\node[MyTitle] at (-1.6, 2.1) {\small Step 1};

\node[point,label={above:1}] (v0) at ( 0.0, 0.0) {};
\node[point,label={above:2}] (v1) at ( 3.0, 0.0) {};
\node[fullpoint,label={below:3}] (v2) at ( 1.5, 1.5) {};
\node[point,label={above:$s$}] (vs) at (-1.5, 0.0) {};

\draw 
    (v0) -- node[below]{\small $\TravelRange / 3$} (v1)
    (v0) -- node[above left]{\small $\TravelRange / 3$} (v2)
    (v1) -- node[above right]{\small $\TravelRange / 3$} (v2)
    (v0) -- node[above]{\small $0$} (vs)
;

\node[MyLabel] at ($ (v0) + (-0.0,-0.2) $) {\footnotesize \underline{$(0, 0, 0, 0, \infty)$}};

\end{scope}


\begin{scope}[shift={(7,0)}]

\node[MyTitle] at (-1.6, 2.1) {\small Step 2};

\node[point,label={above:1}] (v0) at ( 0.0, 0.0) {};
\node[point,label={above:2}] (v1) at ( 3.0, 0.0) {};
\node[fullpoint,label={below:3}] (v2) at ( 1.5, 1.5) {};
\node[point,label={above:$s$}] (vs) at (-1.5, 0.0) {};

\draw 
    (v0) -- node[below]{\small $\TravelRange / 3$} (v1)
    (v0) -- node[above left]{\small $\TravelRange / 3$} (v2)
    (v1) -- node[above right]{\small $\TravelRange / 3$} (v2)
    (v0) -- node[above]{\small $0$} (vs)
;

\node[MyLabel] at ($ (v1) + (-0.0,-0.2) $) {\footnotesize \underline{$(0, 1, \TravelRange / 3, \TravelRange / 3, \infty)$}};
\node[MyLabelU] at ($ (v2) + ( 0.1, 0.1) $) {\footnotesize $(1, 0, \TravelRange / 3, 0, \TravelRange / 3)$};

\end{scope}


\begin{scope}[shift={(0,-5.0)}]

\node[MyTitle] at (-1.6, 2.1) {\small Step 3};

\node[point,label={above:1}] (v0) at ( 0.0, 0.0) {};
\node[point,label={above:2}] (v1) at ( 3.0, 0.0) {};
\node[fullpoint,label={below:3}] (v2) at ( 1.5, 1.5) {};
\node[point,label={above:$s$}] (vs) at (-1.5, 0.0) {};

\draw 
    (v0) -- node[below]{\small $\TravelRange / 3$} (v1)
    (v0) -- node[above left]{\small $\TravelRange / 3$} (v2)
    (v1) -- node[above right]{\small $\TravelRange / 3$} (v2)
    (v0) -- node[above]{\small $0$} (vs)
;

\node[MyLabel] at ($ (v0) + (-0.0,-0.2) $) {\footnotesize \underline{$(0, 1, 2 \TravelRange / 3, 2 \TravelRange / 3, \infty)$}};
\node[MyLabelU] at ($ (v2) + ( 0.1, 0.1) $) {\footnotesize $(1, 0, \TravelRange / 3, 0, \TravelRange / 3)$ \\ \footnotesize $(1, 1, 2 \TravelRange / 3, 0, 2 \TravelRange / 3)$};

\end{scope}


\begin{scope}[shift={(7,-5.0)}]

\node[MyTitle] at (-1.6, 2.1) {\small Step 4};

\node[point,label={above:1}] (v0) at ( 0.0, 0.0) {};
\node[point,label={above:2}] (v1) at ( 3.0, 0.0) {};
\node[fullpoint,label={below:3}] (v2) at ( 1.5, 1.5) {};
\node[point,label={above:$s$}] (vs) at (-1.5, 0.0) {};

\draw 
    (v0) -- node[below]{\small $\TravelRange / 3$} (v1)
    (v0) -- node[above left]{\small $\TravelRange / 3$} (v2)
    (v1) -- node[above right]{\small $\TravelRange / 3$} (v2)
    (v0) -- node[above]{\small $0$} (vs)
;

\node[MyLabel] at ($ (v0) + (-0.0,-0.2) $) {\footnotesize};
\node[MyLabel] at ($ (v1) + (-0.0,-0.2) $) {\footnotesize $(0, 1, \TravelRange, \TravelRange, \infty)$};
\node[MyLabelU] at ($ (v2) + ( 0.1, 0.1) $) {\footnotesize $(1, 0, \TravelRange / 3, 0, \TravelRange / 3)$ \\ \footnotesize \underline{$(1, 1, 2 \TravelRange / 3, 0, 2 \TravelRange / 3)$} \\ \footnotesize $(1, 1, \TravelRange, 0, \TravelRange)$};

\end{scope}


\begin{scope}[shift={(0,-10.0)}]

\node[MyTitle] at (-1.6, 2.1) {\small Step 5};

\node[point,label={above:1}] (v0) at ( 0.0, 0.0) {};
\node[point,label={above:2}] (v1) at ( 3.0, 0.0) {};
\node[fullpoint,label={below:3}] (v2) at ( 1.5, 1.5) {};
\node[point,label={above:$s$}] (vs) at (-1.5, 0.0) {};

\draw 
    (v0) -- node[below]{\small $\TravelRange / 3$} (v1)
    (v0) -- node[above left]{\small $\TravelRange / 3$} (v2)
    (v1) -- node[above right]{\small $\TravelRange / 3$} (v2)
    (v0) -- node[above]{\small $0$} (vs)
;

\node[MyLabel] at ($ (v0) + (-0.0,-0.2) $) {\footnotesize \underline{$(1, 1, \TravelRange, \TravelRange / 3, 2 \TravelRange / 3)$}};
\node[MyLabel] at ($ (v1) + (-0.0,-0.2) $) {\footnotesize $(0, 1, \TravelRange, \TravelRange, \infty)$ \\ \footnotesize $(1, 1, \TravelRange, \TravelRange / 3, 2 \TravelRange / 3)$};
\node[MyLabelU] at ($ (v2) + ( 0.1, 0.1) $) {\footnotesize $(1, 0, \TravelRange / 3, 0, \TravelRange / 3)$ \\ \footnotesize $(1, 1, \TravelRange, 0, \TravelRange)$};

\end{scope}


\begin{scope}[shift={(7,-10.0)}]

\node[MyTitle] at (-1.6, 2.1) {\small Step 6};

\node[point,label={above:1}] (v0) at ( 0.0, 0.0) {};
\node[point,label={above:2}] (v1) at ( 3.0, 0.0) {};
\node[fullpoint,label={below:3}] (v2) at ( 1.5, 1.5) {};
\node[point,label={above:$s$}] (vs) at (-1.5, 0.0) {};

\draw 
    (v0) -- node[below]{\small $\TravelRange / 3$} (v1)
    (v0) -- node[above left]{\small $\TravelRange / 3$} (v2)
    (v1) -- node[above right]{\small $\TravelRange / 3$} (v2)
    (v0) -- node[above]{\small $0$} (vs)
;

\node[MyLabel] at ($ (v0) + ( 0.0,-0.2) $) {};
\node[MyLabel] at ($ (v1) + ( 0.0,-0.2) $) {\footnotesize $(0, 1, \TravelRange, \TravelRange, \infty)$ \\ \footnotesize $(1, 1, \TravelRange, \TravelRange / 3, 2 \TravelRange / 3)$};
\node[MyLabelU] at ($ (v2) + ( 0.1, 0.1) $) {\footnotesize $(1, 0, \TravelRange / 3, 0, \TravelRange / 3)$};
\node[MyLabel] at ($ (vs) + ( 0.0,-0.2) $) {\footnotesize $(1, 1, \TravelRange, \TravelRange / 3, 2 \TravelRange / 3)$};

\end{scope}

\end{tikzpicture}
\vspace{2em}
\caption{An illustrative example of a 3-node transportation network and a corresponding execution sequence of the labeling algorithm. Node 1 is the origin, and node 2 is the destination. Suppose both the travel range and the maximum travel distance are $\TravelRange$ (corresponding to a 50\% deviation tolerance). Open nodes indicate locations without charging stations, whereas closed nodes indicate locations with a charging station. Labels are shown at their respective nodes, and underlined labels denote those selected for extension in subsequent iterations.}
\label{fig_frlp_labelling}
\end{figure}

\subsection{Numerical Experiments}\label{subsec_numerical_experiments}

In this section, numerical experiments are presented to demonstrate the effect of the flexibility introduced by the cyclic FRLP.
In Section~\ref{subsec_frlp_for_maximizing_served_demand}, we run experiments to see how much demand the solution of the original and cyclic FRLP can cover.
In Section~\ref{subsec_frlp_for_serving_all_demands}, we use the original and cyclic FRLP to determine the minimum number of charging stations required to serve all demands.

\subsubsection{FRLP for Maximizing Served Demand}\label{subsec_frlp_for_maximizing_served_demand}

In this section, we use the original and cyclic FRLP to find placements of charging stations to serve as much demand as possible.
We assume that drivers are willing to use different paths for their outbound and return trips and deviate from their shortest cycle up to the factor of $\DeviationTolerance$, which is a parameter varied in the experiments.
In the first setup, we run the original FRLP, which assumes drivers must use the same path for their outbound and inbound trips, and determine the placement of 5 charging stations.
We then reevaluate the amount of served demand under the relaxed assumption that drivers are willing to use different paths on their outbound and inbound trips.
The reevaluated served demand must be greater than or equal to the original FRLP’s objective value.
In the second setup, we run the cyclic FRLP and find the allocation of 5 charging stations that maximizes the amount of served demand, correctly assuming drivers are willing to use different paths in the two directions.
The original FRLP is solved with the method of \citet{ArslanEtAl2019} with formulation~\eqref{eq_agg_extended_formulation}, while the cyclic FRLP is solved with the method described in Section~\ref{subsec_solution_method}.
We use two transportation networks introduced by~\citet{simchi1988heuristic} (see also~\citet{kuby2007location}) and by~\citet{arslan2014impacts} (which is also used by \citet{ArslanEtAl2019} and \citet{GopfertAndBock2019}).
See Table~\ref{tab_instance_stats} for the statistics.
All methods are implemented in C++ with CPLEX 22.1 (executed on a single thread) and executed on a machine with 2 ``Intel(R) Xeon(R) Gold 6226 CPU'' (12 cores each) and 384GB of RAM.

\begin{table}[htbp]
    \centering
    \caption{Characteristics of the transportation networks used in the experiments.}
    \label{tab_instance_stats}
    \begin{tabular}{ccrrr}
\toprule
Name & Reference & Number of nodes & Number of edges & Number of OD pairs \\
\midrule
\textsc{Network25} & \cite{simchi1988heuristic} & 25 & 42 & 300 \\
\textsc{California} & \cite{arslan2014impacts} & 339 & 1234 & 1167 \\
\bottomrule
    \end{tabular}
\end{table}

Table~\ref{tab_max_cover_result} shows the results.
For the original FRLP, it reports the objective value based on the formulation~\eqref{eq_agg_extended_formulation}, and the actual demand served by the solution output from the original FRLP, reevaluated using the assumption of the cyclic FRLP.
For the cyclic FRLP, the objective value equals the amount of served demand, since it correctly calculates the objective value under our assumption.

On the smaller network \textsc{Network25}, the difference in served demand between the two methods is small, except when the deviation tolerance is $\DeviationTolerance = 1.5$.
We observe larger discrepancies on the larger network \textsc{California}.
However, the time required by the cyclic FRLP tends to be larger.
For both methods, the solution time tends to increase as the deviation tolerance $\DeviationTolerance$ gets larger.
We note that the increase is more rapid for the cyclic FRLP compared to the original FRLP.
In fact, the majority of the computational time is spent on the separation procedure, indicating that the labeling algorithm takes longer to execute as the deviation tolerance increases. When solving the original FRLP, most of the computational time is also devoted to separation. Since, as discussed in Section~\ref{subsec_formulations_of_the_cyclic_frlp}, in the original FRLP the separation reduces to a shortest path problem, whereas for the cyclic FRLP the same procedure is not applicable, the higher solving times for the cyclic FRLP are to be expected.
More detailed results are provided in Appendix~\ref{sec_detailed_results}.

\begin{table}[htbp]
    \centering
    \caption{Comparison of original and cyclic FRLP models for maximizing served demand.
Columns indicate deviation tolerance ($\alpha$), original FRLP objective value (Objective), served demand assuming drivers may use different paths outbound and inbound (Served demand), and computation times in seconds (Time).}
    \label{tab_max_cover_result}
    \begin{tabular}{ccrrrrr}
\toprule
&& \multicolumn{3}{c}{Original FRLP} & \multicolumn{2}{c}{Cyclic FRLP} \\
\cmidrule(lr){3-5}\cmidrule(lr){6-7}
& $\DeviationTolerance$ & Objective & Served demand & Time (s) & Served demand & Time (s) \\
\midrule
\multirow[m]{6}{*}{\textsc{Network25}} & 1.0 & 56.16 & 56.16 & 0.10 & 56.16 & 0.04 \\
 & 1.1 & 58.38 & 58.38 & 0.03 & 58.38 & 0.05 \\
 & 1.2 & 59.20 & 59.20 & 0.03 & 59.20 & 0.09 \\
 & 1.3 & 60.12 & 60.77 & 0.04 & 60.77 & 0.23 \\
 & 1.4 & 61.70 & 63.14 & 0.05 & 63.17 & 0.27 \\
 & 1.5 & 63.17 & 65.41 & 0.06 & 68.21 & 0.35 \\
\cmidrule(lr){1-7}
\multirow[m]{6}{*}{\textsc{California}} & 1.0 & 79.94 & 79.94 & 3.32 & 79.94 & 0.82 \\
 & 1.1 & 85.91 & 86.71 & 19.66 & 87.89 & 160.43 \\
 & 1.2 & 89.08 & 91.80 & 24.31 & 92.42 & 356.69 \\
 & 1.3 & 91.61 & 94.37 & 11.21 & 95.31 & 467.08 \\
 & 1.4 & 92.95 & 95.85 & 29.58 & 97.04 & 695.10 \\
 & 1.5 & 94.51 & 96.70 & 29.18 & 97.96 & 865.44 \\
\bottomrule
    \end{tabular}
\end{table}

\subsubsection{FRLP for Serving All Demands}\label{subsec_frlp_for_serving_all_demands}

Although we have focused on the FRLP for maximizing the amount of served demand, formulation~\eqref{eq_agg_extended_formulation} can be modified to solve the FRLP with other objectives.
For example, as discussed by \citet{ArslanEtAl2019}, the FRLP that finds the minimum number of charging stations to serve all demands can be formulated as
\begin{alignat}{2}
\min_{\MasterVar \in \{0, 1\}^{|\NodeSet|}} \ & \sum_{\NodeIdx \in \NodeSet} x_{\NodeIdx}
\label{eq_frlp_all_cover_agg_extended_formulation}
\\
\text{s.t.} \ & \sum_{\NodeIdx \in \NodeCut} \MasterVar_{\NodeIdx} \ge 1, && \qquad \forall \DemandIdx \in \DemandSet, \NodeCut \in \AggNodeCutSet{\DemandIdx}. 
\notag
\end{alignat}
The adaptation of the solution method to solve this variant is straightforward, and hence omitted.

In the next experiment, we run the original FRLP and the cyclic FRLP to find the minimum number of charging stations required to serve all demands by solving~\eqref{eq_frlp_all_cover_agg_extended_formulation}.
Table~\ref{tab_all_cover_result} shows the results.
Similar to the FRLP for maximizing served demands, it takes more time to solve the cyclic FRLP than the original FRLP.
However, the cyclic FRLP tends to lead to solutions using fewer charging stations than the original FRLP, especially when the deviation tolerance $\DeviationTolerance$ is large.
For both the original FRLP and cyclic FRLP, the solution time grows as the deviation tolerance $\DeviationTolerance$ increases.
Therefore, it is important to use the correct variant of the FRLP depending on driver behavior.
Additional results are again provided in Appendix~\ref{sec_detailed_results}.

\begin{table}[htbp]
    \centering
    \caption{Comparison of original and cyclic FRLP models for serving all demands with the minimum number of charging stations.}
    \label{tab_all_cover_result}
    \begin{tabular}{ccrrrr}
\toprule
&& \multicolumn{2}{c}{Original FRLP} & \multicolumn{2}{c}{Cyclic FRLP} \\
\cmidrule(lr){3-4}\cmidrule(lr){5-6}
& $\DeviationTolerance$ & Num.\ of charging stations & Time (s) & Num.\ of charging stations & Time (s) \\
\midrule
\multirow[m]{6}{*}{\textsc{Network25}} & 1.0 & 17 & 0.02 & 17 & 0.02 \\
 & 1.1 & 17 & 0.02 & 17 & 0.04 \\
 & 1.2 & 17 & 0.02 & 17 & 0.08 \\
 & 1.3 & 17 & 0.02 & 16 & 0.12 \\
 & 1.4 & 17 & 0.03 & 16 & 0.16 \\
 & 1.5 & 16 & 0.03 & 14 & 0.19 \\
\cmidrule(lr){1-6}
\multirow[m]{6}{*}{\textsc{California}} & 1.0 & 24 & 6.14 & 24 & 1.40 \\
 & 1.1 & 22 & 10.66 & 19 & 43.12 \\
 & 1.2 & 19 & 14.57 & 16 & 106.23 \\
 & 1.3 & 18 & 17.55 & 12 & 160.47 \\
 & 1.4 & 16 & 18.86 & 9 & 212.63 \\
 & 1.5 & 14 & 19.79 & 9 & 279.44 \\
\bottomrule
    \end{tabular}
\end{table}

\section{Conclusions}\label{sec_conclusions}

We studied the FRLP, an optimization problem for determining the optimal placement of refueling stations for vehicles with limited travel ranges, such as EV charging stations.
Specifically, we examined two formulations from the literature, one developed independently by \citet{ArslanEtAl2019} and \citet{GopfertAndBock2019}, and another one proposed by \citet{GopfertAndBock2019}.
The former formulation is used in the state-of-the-art algorithm, which is based on a branch-and-cut approach.
We showed that this formulation is tighter than the latter and that the relative difference in bounds obtained from their LP relaxation can be arbitrarily large.
This tightness may partially explain the superior performance of the current state-of-the-art solution method.

We also addressed a common routing assumption used in the literature of the FRLP, which requires drivers to use the same path for their outbound and inbound trips.
To relax this, we introduced a variant called the cyclic FRLP, where drivers may take different paths on their outbound and return trips.
We demonstrated how existing formulations for the original FRLP can be naturally extended to model this variant and presented its solution method.
Furthermore, we presented a labeling algorithm for finding a traversable cycle, enabling the use of a branch-and-cut approach to solve the cyclic FRLP.
Through a small example and numerical experiments on the Californian network from \citet{arslan2014impacts}, we highlighted the significant impact of correctly modeling driver behavior on solution quality.

The cyclic FRLP is more flexible than the original FRLP, which requires drivers to use the same path for both outbound and return trips.
It relaxes this assumption by allowing different paths in each direction, but still assumes that the same cycle is used for every round trip.
A natural extension is to allow drivers to repeat an $n$-sequence of cycles with some predetermined number $n$, which could model medium-distance travel where refueling or recharging is not needed after every trip.
The solution method, including the separation procedure based on the labeling algorithm, can be readily adapted to handle this extension.
However, evaluating its practical relevance is left for future work.

\appendix

\section{Detailed Results}\label{sec_detailed_results}

This section contains detailed results.
Table~\ref{tab_max_cover_result_detailed} corresponds to the experiment in Section~\ref{subsec_frlp_for_maximizing_served_demand} (the FRLP for maximizing the served demand).
Table~\ref{tab_all_covering_result_detailed} is the result of the experiment in Section~\ref{subsec_frlp_for_serving_all_demands} (the FRLP for covering all demands).

\begin{table}[htbp]
\centering
\caption{Comparison of original and cyclic FRLP models for maximizing served demand.
Columns indicate deviation tolerance ($\alpha$), total computation times in seconds (Time), computation times spent in separation (Separation time), the number of branch-and-bound nodes (\# B\&B nodes), and the number of violated constraints added (\# cuts).}
    \label{tab_max_cover_result_detailed}
\begin{tabular}{lllrrrr}
\toprule
Instance & Routing & $\DeviationTolerance$ & Time (s) & Separation time (s) & Num.\ of B\&B nodes & Num.\ of cuts \\
\midrule
\multirow[m]{12}{*}{\textsc{network25}} & \multirow[m]{6}{*}{original} & 1.0 & 0.10 & 0.01 & 7 & 721 \\
 &  & 1.1 & 0.03 & 0.01 & 3 & 586 \\
 &  & 1.2 & 0.03 & 0.02 & 8 & 636 \\
 &  & 1.3 & 0.04 & 0.02 & 7 & 673 \\
 &  & 1.4 & 0.05 & 0.03 & 5 & 727 \\
 &  & 1.5 & 0.06 & 0.04 & 3 & 834 \\
\cmidrule(lr){2-7}
 & \multirow[m]{6}{*}{cyclic} & 1.0 & 0.04 & 0.01 & 15 & 761 \\
 &  & 1.1 & 0.05 & 0.03 & 11 & 666 \\
 &  & 1.2 & 0.09 & 0.07 & 5 & 722 \\
 &  & 1.3 & 0.23 & 0.21 & 17 & 955 \\
 &  & 1.4 & 0.27 & 0.24 & 10 & 848 \\
 &  & 1.5 & 0.35 & 0.33 & 12 & 858 \\
\cmidrule(lr){1-7}
\multirow[m]{12}{*}{\textsc{California}} & \multirow[m]{6}{*}{original} & 1.0 & 3.32 & 3.28 & 0 & 1448 \\
 &  & 1.1 & 19.66 & 19.59 & 7 & 2178 \\
 &  & 1.2 & 24.31 & 24.25 & 3 & 2116 \\
 &  & 1.3 & 11.21 & 11.18 & 0 & 1390 \\
 &  & 1.4 & 29.58 & 29.51 & 5 & 1906 \\
 &  & 1.5 & 29.18 & 29.13 & 2 & 2026 \\
\cmidrule(lr){2-7}
 & \multirow[m]{6}{*}{cyclic} & 1.0 & 0.82 & 0.77 & 0 & 2134 \\
 &  & 1.1 & 160.43 & 160.26 & 32 & 3185 \\
 &  & 1.2 & 356.69 & 356.44 & 45 & 2710 \\
 &  & 1.3 & 467.08 & 466.91 & 39 & 3025 \\
 &  & 1.4 & 695.10 & 694.90 & 130 & 2918 \\
 &  & 1.5 & 865.44 & 865.33 & 0 & 2592 \\
\bottomrule
\end{tabular}
\end{table}

\begin{table}[htbp]
\centering
\caption{Comparison of original and cyclic FRLP models for serving all demands with the minimum number of charging stations.}
    \label{tab_all_covering_result_detailed}
\begin{tabular}{lllrrrr}
\toprule
Instance & Routing & $\DeviationTolerance$ & Time (s) & Callback time (s) & Num.\ of B\&B nodes & Num.\ of cuts \\
\midrule
\multirow[m]{12}{*}{\textsc{network25}} & \multirow[m]{6}{*}{original} & 1.0 & 0.02 & 0.01 & 3 & 945 \\
 &  & 1.1 & 0.02 & 0.01 & 3 & 934 \\
 &  & 1.2 & 0.02 & 0.02 & 2 & 929 \\
 &  & 1.3 & 0.02 & 0.02 & 2 & 926 \\
 &  & 1.4 & 0.03 & 0.02 & 4 & 928 \\
 &  & 1.5 & 0.03 & 0.02 & 3 & 924 \\
\cmidrule(lr){2-7}
 & \multirow[m]{6}{*}{cyclic} & 1.0 & 0.02 & 0.02 & 0 & 901 \\
 &  & 1.1 & 0.04 & 0.04 & 0 & 902 \\
 &  & 1.2 & 0.08 & 0.08 & 3 & 901 \\
 &  & 1.3 & 0.12 & 0.12 & 3 & 907 \\
 &  & 1.4 & 0.16 & 0.16 & 3 & 906 \\
 &  & 1.5 & 0.19 & 0.18 & 2 & 900 \\
\cmidrule(lr){1-7}
\multirow[m]{12}{*}{\textsc{California}} & \multirow[m]{6}{*}{original} & 1.0 & 6.14 & 6.12 & 24 & 3712 \\
 &  & 1.1 & 10.66 & 10.64 & 14 & 3535 \\
 &  & 1.2 & 14.57 & 14.56 & 16 & 3501 \\
 &  & 1.3 & 17.55 & 17.54 & 8 & 3554 \\
 &  & 1.4 & 18.86 & 18.84 & 13 & 3513 \\
 &  & 1.5 & 19.79 & 19.77 & 1 & 3517 \\
\cmidrule(lr){2-7}
 & \multirow[m]{6}{*}{cyclic} & 1.0 & 1.40 & 1.38 & 73 & 3903 \\
 &  & 1.1 & 43.12 & 43.10 & 7 & 3509 \\
 &  & 1.2 & 106.23 & 106.20 & 28 & 3517 \\
 &  & 1.3 & 160.47 & 160.46 & 0 & 3554 \\
 &  & 1.4 & 212.63 & 212.62 & 0 & 3534 \\
 &  & 1.5 & 279.44 & 279.41 & 6 & 3576 \\
\bottomrule
\end{tabular}
\end{table}

\section*{Data Availability}

Data available on request from the authors.

\section*{Acknowledgements}

This research was supported by Hydro-Québec, NSERC Collaborative Research and Development Grant CRDPJ 536757 - 19.

The authors are grateful to Professor Okan Arslan, who kindly shared the Californian network used in our numerical experiments.

\section*{Declarations}
\textbf{Conflict of interest} The authors declare that they have no conflict of interest.

\bibliography{references}

\end{document}